\documentclass[hidelinks,12pt]{article}

\usepackage[english]{babel}        
\usepackage[all]{xy}
\usepackage[normalem]{ulem}
\CompileMatrices

\usepackage[utf8]{inputenc}
\usepackage[T1]{fontenc}
\usepackage{multicol}
\usepackage{amscd, amsfonts, amsmath, amssymb, amsthm, array, authblk, caption, color, enumerate, enumitem, float, framed, geometry, graphicx, hyperref, latexsym, lmodern, mathrsfs, pb-diagram, pdflscape, rotating, slantsc, stmaryrd, subcaption, tabu, tikz, times, url, verbatim}

\usepackage{xcolor}
\hypersetup{
    colorlinks,
    linkcolor={red!50!black},
    citecolor={blue!50!black},
    urlcolor={blue!80!black}
}

\newtheorem{theorem}{Theorem}[section]
\newtheorem{lemma}[theorem]{Lemma}
\newtheorem{proposition}[theorem]{Proposition}
\newtheorem{corollary}[theorem]{Corollary}

\theoremstyle{definition}
\newtheorem{definition}[theorem]{Definition}
\newtheorem{example}[theorem]{Example}
\newtheorem{notation}[theorem]{Notation}

\newtheorem{remark}[theorem]{Remark}

\newcommand{\B}[1]{\textit{T}_{#1}}

\DeclareMathOperator{\Zx}{\mathbb{Z}(\textit{x})}

\DeclareMathOperator{\Zxy}{\mathbb{Z}(\textit{x},\textit{y})}

\newcommand{\U}[1]{\textit{U}_{#1}\left(\frac{1}{x}\right)}
\newcommand{\gl}[1]{GL_{#1}(\Lambda)}

\newcommand{\gen}[1]{\textit{t}_{#1}}
\newcommand\scalemath[2]{\scalebox{#1}{\mbox{\ensuremath{\displaystyle #2}}}}

\newcommand{\subscript}[2]{$#1  #2$}

\numberwithin{equation}{section}

\newcommand\risSpdf[6]{\raisebox{#1pt}[#5pt][#6pt]{\begin{picture}(#4,15)(0,0)		
  \put(0,0){\includegraphics[width=#4pt]{#2.pdf}} #3
     \end{picture}}}
     
\newcommand\risS[6]{\raisebox{#1}[#5pt][#6pt]{\begin{picture}(#4,15)(0,0)		
  \put(0,0){\includegraphics[width=#4pt]{#2.pdf}} #3
     \end{picture}}}

\newcounter{braid}
\newcounter{strands}
\pgfkeyssetvalue{/tikz/braid height}{.5cm}
\pgfkeyssetvalue{/tikz/braid width}{.5cm}
\pgfkeyssetvalue{/tikz/braid start}{(0,0)}
\pgfkeyssetvalue{/tikz/braid colour}{black}
\pgfkeys{/tikz/strands/.code={\setcounter{strands}{#1}}}

\makeatletter
\def\cross{%
  \@ifnextchar^{\message{Got sup}\cross@sup}{\cross@sub}}

\def\cross@sup^#1_#2{\render@cross{#2}{#1}}
\def\cross@sub_#1{\@ifnextchar^{\cross@@sub{#1}}{\render@cross{#1}{1}}}
\def\cross@@sub#1^#2{\render@cross{#1}{#2}}

\def\render@cross#1#2{
  \def\strand{#1}
  \def\crossing{#2}
  \pgfmathsetmacro{\cross@y}{-\value{braid}*\braid@h}
  \pgfmathtruncatemacro{\nextstrand}{#1+1}
  \foreach \thread in {1,...,\value{strands}}
  {
    \pgfmathsetmacro{\strand@x}{\thread * \braid@w}
    \ifnum\thread=\strand
    \pgfmathsetmacro{\over@x}{\strand * \braid@w + .5*(1 - \crossing) * \braid@w}
    \pgfmathsetmacro{\under@x}{\strand * \braid@w + .5*(1 + \crossing) * \braid@w}
    \draw[braid] \pgfkeysvalueof{/tikz/braid start} +(\under@x pt,\cross@y pt) to[out=-90,in=90] +(\over@x pt,\cross@y pt -\braid@h);
    \draw[braid] \pgfkeysvalueof{/tikz/braid start} +(\over@x pt,\cross@y pt) to[out=-90,in=90] +(\under@x pt,\cross@y pt -\braid@h);
    \else
    \ifnum\thread=\nextstrand
    \else
     \draw[braid] \pgfkeysvalueof{/tikz/braid start} ++(\strand@x pt,\cross@y pt) -- ++(0,-\braid@h);
    \fi
   \fi
  }
  \stepcounter{braid}
}

\tikzset{braid/.style={double=\pgfkeysvalueof{/tikz/braid colour},double distance=1pt,line width=0pt,white}} 

\newcommand{\braid}[2][]{%
  \begingroup
  \pgfkeys{/tikz/strands=2}
  \tikzset{#1}
  \pgfkeysgetvalue{/tikz/braid width}{\braid@w}
  \pgfkeysgetvalue{/tikz/braid height}{\braid@h}
  \setcounter{braid}{0}
  \let\sigma=\cross
  #2
  \endgroup
}
\makeatother

\usetikzlibrary{decorations}
\usetikzlibrary{decorations.pathmorphing,shapes,calc,arrows,through,intersections,decorations.pathreplacing}
\newdimen{\tempx}
\newdimen{\tempy}

\tikzset{
    every picture/.style={
        remember picture,   
        inner xsep=0pt, 
        inner ysep=1pt, 
        baseline,       
        every node/.style={
            anchor=base 
        }
    }
}

\begin{document}

\title{An Alexander type invariant for doodles}
\author{B. Cisneros, M. Flores, J. Juyumaya and C. Roque-Márquez}
\date{\empty}

\maketitle

\begin{abstract}
We construct an Alexander type invariant for oriented doodles from a deformation of the Tits representation of the twin group and from the Chebyshev polynomials of second kind. Similar to the Alexander polynomial, our invariant vanishes on unlinked doodles with more than one component. We also include values of our invariant on several doodles.
\end{abstract}



\section{Introduction}

Doodles were firstly introduced by Fenn and Taylor \cite{FT79} as a finite collection of embedded circles on the 2-sphere with no triple intersections. Later, Khovanov \cite{Kh97} pointed out that it was more natural to consider doodles as immersed circles on the 2-sphere than as simple curves. In this manner, some concepts of knot theory were successfully transferred to doodle theory. In particular, he obtained an associated group to a doodle, called the doodle group, in analogue to the fundamental group of the complement of a link. 
In contrast with knot theory, equivalence of doodles just considers a planar version of the first and the second Reidemeister moves (see Figure \ref{fig:Reidemeister}); the lack of the third Reidemeister move leads to significant differences in comparison with classical knot theory. For instance, every class of a doodle has a unique representative with minimal number of crossings \cite{Kh97}; the calculation of Vassiliev invariants is more complicated than classical knots, but Vassiliev invariants classify doodles, problem that remains open for knots \cite{Mer03}. Extended to a more general setting, recent work by Bartholomew-Fenn-Kamada-Kamada \cite{BFKK18} consider doodles on closed oriented surfaces of any genus, which can be considered as virtual links analogue for doodles; in \cite{BFKK18-gauss} they give a complete invariant for virtual doodles\footnote{In Bartholomew's webpage \cite{BarWeb} is shown a table of virtual doodles up to 10 crossings.}.

The algebraic counterparts of doodles are the so-called \textit{twin groups}, terminology due to Khovanov. The role of these groups in doodle theory is similar to that of Artin's braid groups in knot theory. These groups have appeared under different names and contexts in the literature: {cartographical Grothendieck groups} \cite{Voe90}, {quantum symmetric groups} \cite{GGS92}, {flat braid groups}\footnote{This term was later used in \cite{KauLam04} to an entirely different object} \cite{Mer99}, {traid groups} \cite{HK20} and {planar braid groups} \cite{GoMeRo19,MosRo19}. 
Further, the twin group on $n$ strands, denoted by $\B{n}$, is a right angled Coxeter group generated by $n-1$ involutions and its elements can be depicted as planar braids on $n$ strands. The permutation induced by a planar braid defines a natural epimorphism from $\B{n}$ to the symmetric group on $n$ symbols, whose kernel is called the \textit{pure twin group}. So far, the twin group and pure twin group have been studied in the last years, though there still exist many open questions about these groups. For details on recent developments, refer to \cite{BSV18,DG18,GoMeRo19,HK20,KNS19-2,KNS19,MosRo19}. 

The first evidence of the close relationship between doodles and twins is by the straightforward process of obtaining a doodle from a given twin by `closing' the twin. In the reverse process, Khovanov \cite{Kh97} proved the Alexander theorem for doodles in which any doodle on the 2-sphere is the closure of a twin, and recently Gotin \cite{Got18} has established the Markov theorem for doodles on the 2-sphere in which he defines the corresponding Markov moves for twins. These two theorems invite to the construction of doodle invariants in a similar way it has been done in classical knot theory. For instance, the construction of an Alexander and Jones polynomials for doodles.

The purpose of this paper is the construction of an Alexander type invariant for doodles. We follow Burau's work to compute the Alexander polynomial for links via certain representation of the braid group \cite{Bur35}, applied in the context of twins and doodles.

More precisely, we define a representation $\psi_n:\B{n}\to GL_{n-1}(\mathbb{Z}(x))$ which is a deformation of the Tits representation of $\B{n}$. Then, using $\psi_n$ we construct a function $f_n:\B n\to \Zx$ which defines an invariant of doodles up to a factor $x^{2k}$, i.e., if two twins $\alpha$, $\beta$ have equivalent closures $\hat{\alpha} \sim \hat{\beta}$ as doodles, then their images $f_n(\alpha)$,  $f_m(\beta)$ differ by a multiple of $x^{2k}$ for some $k\in \mathbb{Z}$ (see Lemma \ref{lem:hyper-stabil}).
Thus, we obtain an invariant $\mathcal{Q}(D)$ for a doodle $D$ as the smallest degree polynomial over all images under functions $f_n$ of twins with closure $D$ (Theorem \ref{thm:inv Q}).
Furthermore, we verify that these functions satisfy the skein relation
    $$f_n\big( {\risSpdf{-9}{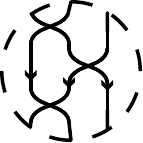}{}{25}{0}{0}} \big) - f_n\big( {\risSpdf{-9}{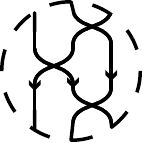}{}{25}{0}{0}} \big)
        = (x^2-1)\left(f_n\big( {\risSpdf{-9}{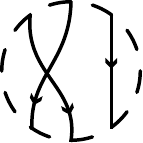}{}{25}{0}{0}} \big) - f_n\big( {\risSpdf{-9}{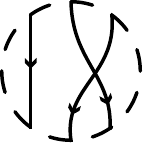}{}{25}{0}{0}} \big) \right).$$
It is worth noting that Chebyshev polynomials of second kind appeared unexpectedly in the construction of the functions $f_n$. Observe that it is not the first time where Chebyshev polynomials of second kind emerge in low dimensional topology. For instance, they play an important role defining Jones-Wenzl projectors \cite{Wang10}.

The paper is organized as follows. In Section \ref{sec:preliminaries}, we give some preliminaries on doodles and twins; we recall the Tits representation of a Coxeter group and list necessary properties of Chebyshev polynomials. In Section \ref{sec:Representation}, we study a deformation of the Tits representation of $\B{n}$ as a Coxeter group and some of its features. In Section \ref{sec:Invariant}, we investigate the behavior of the functions $f_n$ under Markov moves, obtaining the definition of our polynomial invariant $\mathcal{Q}$. Finally, we give some examples of computations in Section \ref{sec:Computations}. 

\subsection*{Acknowledgments}
The authors would like to thank Jacob Mostovoy and Jesús Gon\-zá\-lez for useful conversations. The first author was supported by Cátedras-CONACYT (project no. 61). The second author was partially supported by grant FONDECYT 11\-17\-03\-05. The third author was supported, in part, by grant FONDECYT 1180036. The fourth author was partially supported by a CONACYT Postdoctoral Fellowship. This work is a result of two Mexican-Chilean meetings celebrated at Oaxaca and at Valparaiso founded by FONDECYT 11\-17\-03\-05 and FONDECYT 1180036.

\section{Preliminaries}\label{sec:preliminaries}

In the present section, we give the necessary background on doodles, twins, Tits representation and Chebyshev polynomials that will be used along the paper. This section contains nothing essentially new.

\subsection{Doodles and twin group}

\begin{definition}[cf. Khovanov, \cite{Kh97}]\label{def:doodle}
    A \textit{doodle} is an immersion $D:\bigsqcup_n S^1 \to S^2$ of a disjoint union of $n$ circles into the 2-sphere\footnote{In \cite{BFKK18}, doodles in $S^2$ are called \textit{planar doodles}.} with no triple or higher multiple intersections; $n$ is the number of components. An \textit{oriented doodle} is a doodle in which each component is oriented. We denote the set of oriented doodles by $\mathcal{D}$.
\end{definition}

Throughout this work we will only be concerned with oriented doodles, hence we will omit the word oriented in most of the cases.

Two doodles $D$ and $D'$ are \textit{equivalent}, denoted by $D\sim D'$, if there exists a homotopy between $D$ and $D'$ in which no triple intersections are produced throughout the homotopy. 
Equivalently\footnote{On \textit{regular} doodles, i.e., all its multiple points are transversal double points \cite{BFKK18}.}, $D\sim D'$ if they can be transformed into each other through isotopies of $S^2$ and a finite sequence of local moves R1 and R2, see Figure \ref{fig:Reidemeister}. These moves will be called the first and the second Reidemeister moves.
    
\begin{figure}[h!]
    \begin{minipage}[t]{.45\textwidth}
        \centering
         \includegraphics[scale=0.9]{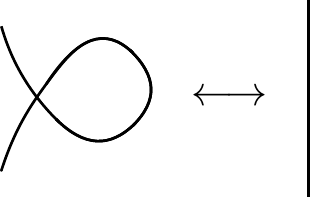}
        \subcaption{R1}\label{fig:R1}
    \end{minipage}
    \hfill
    \begin{minipage}[t]{.45\textwidth}
        \centering
        \includegraphics[scale=0.9]{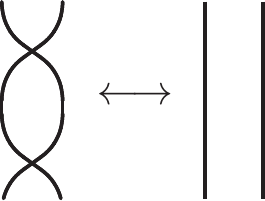}
        \subcaption{R2}\label{fig:R2}
    \end{minipage}
    \caption{Reidemeister moves.} \label{fig:Reidemeister}
\end{figure}

The doodle in Figure \ref{fig:trivial doodle} is called the trivial doodle. The first non-trivial doodle is the Borromean doodle, shown in Figure \ref{fig:borromean}, with six crossings and three components. Following \cite{BFKK18}, the doodle in Figure \ref{fig:4-poppy} is called the 4-poppy doodle and it has eight crossings; it is the first non-trivial 1-component doodle \cite[Theorem~4.2]{BFKK18}.

\begin{figure}[ht]
    \begin{minipage}[t]{.3\textwidth}
        \centering
        \includegraphics{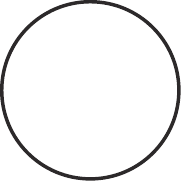}
        \subcaption{Trivial}\label{fig:trivial doodle}
    \end{minipage}
    \hfill
    \begin{minipage}[t]{.3\textwidth}
        \centering
        \includegraphics{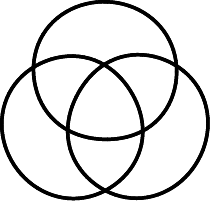}
        \subcaption{Borromean}\label{fig:borromean}
    \end{minipage}
    \hfill
    \begin{minipage}[t]{.3\textwidth}
        \centering
        \includegraphics{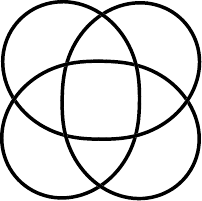}
        \subcaption{4-poppy}\label{fig:4-poppy}
    \end{minipage}
    \caption{First doodles.}\label{fig:First doodles}
\end{figure}

The algebraic counterparts of doodles are twin groups. More precisely, these groups are to doodle theory as braid groups are to knot theory.

\begin{definition}\label{def:presentation twin grp}
    The \textit{twin group} $\B n$ is the group presented by generators $\gen 1,\dots,\gen {n-1}$ and the relations:
\begin{align}
    	\gen i^2 & =  1  & & \text{for all $1 \leq i \leq n-1$}, \label{rel:invol twin grp} \\
		\gen i\gen j & = \gen j\gen i & &\text{for all $1 \leq i,j \leq n-1$ with $ |i-j|>1$}. \label{rel:comm twin grp}
\end{align}
\end{definition}

In \cite{Kh96}, Khovanov gives a geometrical interpretation of $\B n$. To be precise, the elements of $\B n$ can be regarded as planar braids, called twins. Namely, a \textit{twin} on $n$ strands is a collection of $n$ descending arcs in $\mathbb{R}\times[0,1]$ which no three arcs have a point in common. In particular, the generator $\gen i$ is represented by the following ``elementary'' twin:

\begin{equation}
    \begin{tikzpicture}
        \braid[braid colour=black,strands=8,braid start={(0,.5)}, braid height= 1.6cm, braid width = 5mm]
        {\sigma_4};
        \node[font=\scriptsize] at (1.9,-1.5) {$i$};
        \node[font=\scriptsize] at (2.6,-1.5) {$i+1$};
    \end{tikzpicture}
    \, \raisebox{-1cm}{.}
\end{equation}

Relations (\ref{rel:invol twin grp}) and (\ref{rel:comm twin grp}) are represented, respectively, as follows:
\begin{figure}[H]
    \begin{minipage}[t]{.45\textwidth}
        \centering
         \includegraphics[scale=0.8]{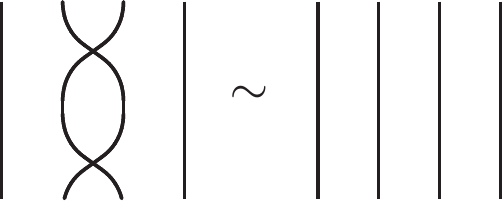}
        \subcaption{Relation (\ref{rel:invol twin grp})}
    \end{minipage}
    \hfill
    \begin{minipage}[t]{.45\textwidth}
        \centering
        \includegraphics[scale=0.8]{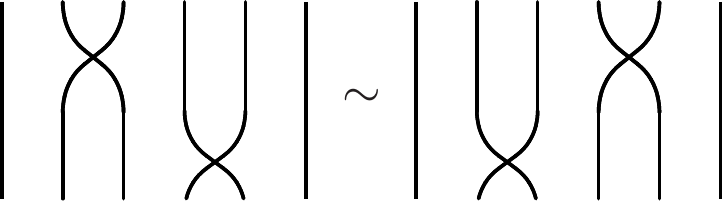}
        \subcaption{Relation (\ref{rel:comm twin grp})}
    \end{minipage}\hspace{-3mm}.
\end{figure}

For $n\geq 1$, let $\iota^R_n:\B n\hookrightarrow \B {n+1}$ be the natural inclusion given by $\iota^R_n(\gen i)=\gen i$, and let $\iota^L_n:\B n\hookrightarrow \B {n+1}$ be the inclusion given by $\iota^L_n(\gen i)=\gen {i+1}$, for $1\leq i \leq n-1$ in both cases. Thus, geometrically $\iota^R_n$ (resp. $\iota^L_n$) adds a vertical strand to the right (resp. left) of the twin, see Figure \ref{fig:inclusion}. 
We denote $\iota^R=\{\iota_n^R \}_{n\geq 1}$ and $\iota^L=\{\iota_n^L \}_{n\geq 1}$ to the systems of inclusions, and we will simply write $\iota^R(\beta)$ for $\iota_n^R(\beta)$ (resp. $\iota^L(\beta)$ for $\iota^L_n(\beta)$) if $\beta \in \B n$ for some $n\geq 1$. Let $T_{\infty}$ denote the inductive limit associated to $\iota^R$.

\begin{figure}[ht]
        \centering
        \includegraphics[scale=1.2]{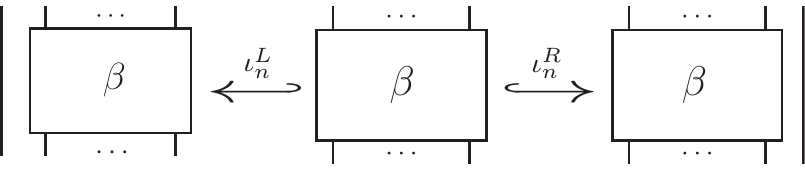}
        \caption{Inclusions $\iota^R_n$ and $\iota^L_n$.}\label{fig:inclusion}
\end{figure}

\subsection{Alexander and Markov theorems for doodles}\label{subsec:Alex-Markov}

The \textit{closure} $\hat{\beta}$ of the twin $\beta$ is the doodle obtained by connecting its upper and lower ends respectively as shown in Figure \ref{fig:borromean as closure}. Khovanov proved the analogue to the classical Alexander theorem for doodles. 

\begin{theorem}[Khovanov, {\cite[Theorem~2.1]{Kh97}}]\label{thm:alexander}
    Any doodle on the 2-sphere is the closure of a twin.
\end{theorem}
\begin{figure}[ht]
    \centering
    \includegraphics[scale=1.3]{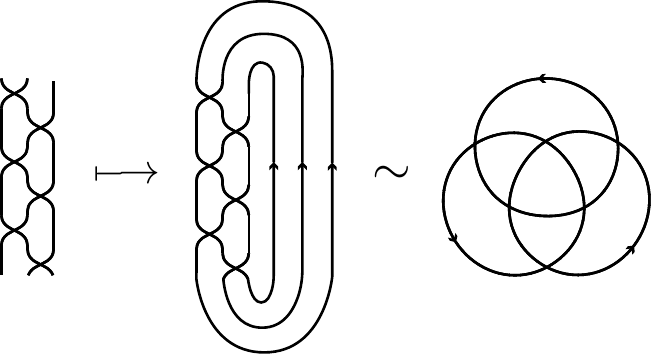}
    \caption{The Borromean doodle as the closure $\widehat{(\gen {1}\gen {2})^3}$.}
    \label{fig:borromean as closure}
\end{figure}

Theorem \ref{thm:alexander} says the map $\beta\mapsto \widehat{\beta}$ from $\B{\infty}$ onto $\mathcal{D}$ is surjective.
However, this map is highly not injective, for instance the twins $\gen{1}\in\B 2$ and $\gen{1}\gen{2}\in\B 3$ have the same closure. Recently in \cite{Got18}, Gotin proves the analogue to the Markov theorem for doodles filling the gap in this correspondence. In order to state this theorem, we set
\begin{align}
    \gen {n,i}&:=\gen {n}\gen {n-1}\cdots \gen {n-i+1}\gen {n-i}\gen {n-i+1} \cdots \gen {n-1}\gen {n}  \label{eq: notation t_n,i},\\
    \gen {1,i}&:=\gen {1}\gen {2}\cdots \gen {i}\gen {i+1}\gen {i} \cdots \gen {2}\gen {1}  \label{eq: notation t_1,i}
\end{align}
for $0\leq i \leq n-1$.

\begin{theorem}[Gotin, {\cite[Theorem~4.1]{Got18}}]\label{thm:Markov}
    Let $\gamma$ and $\gamma'$ be two twins with $D$ and $D'$ their respective closures. Then, $D$ and $D'$ are equivalent doodles if and only if $\gamma$ can be transformed into $\gamma'$ by a finite sequence of the following moves:
    \begin{enumerate}[label=\subscript{M}{{\arabic*}} , itemindent=2mm, start=0]
        \item  : $\iota^R(\beta) \sim \iota^L(\beta)\,$ for any $\beta\in \B n$, \label{rel:M0}
        \item  : $\alpha\beta \sim \beta\alpha \,$ for any $\alpha,\beta\in \B n$, \label{rel:M1}
        \item  : $ \beta \sim \iota^R(\beta) \gen {n,i} \;$ for any $\beta \in \B n$ and $0\leq i \leq n-1$, \label{rel:M2}
        \item  : $\beta \sim \iota^L(\beta) \gen {1,i} \;$ for any $\beta \in \B n$ and $0\leq i \leq n-1$. \label{rel:M3}
    \end{enumerate}
    We call these moves, the Gotin-Markov moves. 
\end{theorem}

\begin{figure}[htb]
    \centering
    \includegraphics[scale=1.3]{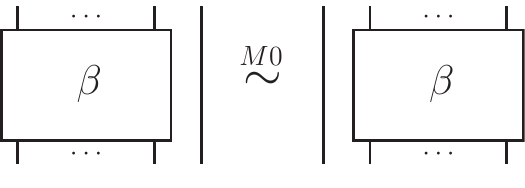}
    \caption{\ref{rel:M0} move.}
    \label{fig:M0 move}
\end{figure}

\begin{figure}[htb]
    \centering
    \includegraphics[scale=1.3]{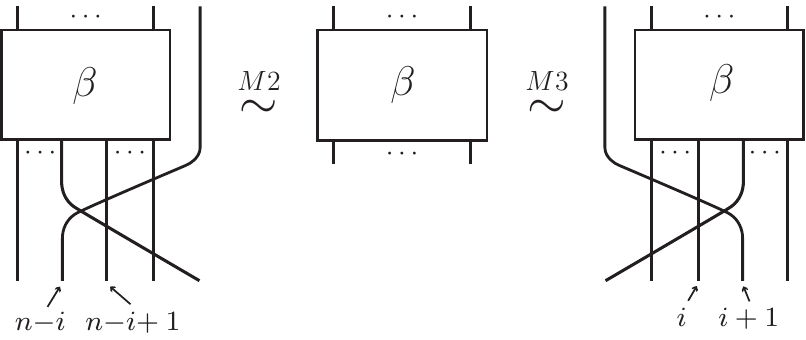}
    \caption{\ref{rel:M2} and \ref{rel:M3} moves.}
    \label{fig:M2 M3}
\end{figure}

\begin{remark}
    When $i=0$ in $M2$ we have $\gen{n,0}=\gen n$, then the $M2$ move becomes the classical second Markov move.
\end{remark}

The Gotin-Markov moves define an equivalence relation on $\B \infty$, denoted by $\sim_G$. Thus, Theorem \ref{thm:alexander} together with Theorem  \ref{thm:Markov} implies the bijection
\begin{equation*}
    \mathcal{D} / \sim \quad \longleftrightarrow  \quad T_{\infty} /\sim_{G}.
\end{equation*}

This bijection says that constructing an invariant for doodles it reduces to find a  family $f=\{f_n: \B n \to \mathcal{S}\}_{n\geq 1}$, where $\mathcal{S}$ is a well understood set and the functions $f_n$ are invariant under the Gotin-Markov moves. In Section \ref{sec:Invariant}, we will construct a such family by taking $\mathcal{S}=\Zx$ and certain $f_n$ defined in terms of a deformation of the Tits representation of the twin group.

\subsection{Tits representation of Coxeter groups}\label{subsec:Coxeter}

We describe briefly some properties of Coxeter groups and the particular case for twin groups. Let $S$ be a finite set; a \textit{Coxeter matrix} on $S$ is a symmetric matrix $M= (m_{s,t})_{s,t \in S}$ such that $m_{s,t}\in \{2,3,\dots \}\cup \{\infty\}$ and $m_{s,s}=1$. The Coxeter group associated to $M$ is the group defined by the following presentation: 
    $$ W(M):= \langle \; S \; | \; (st)^{m_{s,t}} \; \text{for all $s,t$ with $m_{s,t}\neq \infty$ } \rangle . $$

Tits constructed a faithful linear representation for any Coxeter group in the following way. Let $S=\{s_1, s_2, \dots, s_n\}$ be a finite set, $M$ be a Coxeter matrix on $S$ and $V$ be a real vector space with basis $\{e_1,\dots,e_{n}\}$. Now define the symmetric bilinear form $B$ on $V$ by

    $$B(e_i,e_j)= 
        \begin{cases}
        -\cos(\frac{\pi}{m_{i,j}}) & \text{if $m_{s_i,s_j}\neq \infty$}, \\
        -1 & \text{if $m_{s_i,s_j} = \infty$}.
        \end{cases}
    $$
The linear map $\sigma_i : V \to V$ given by $\sigma_i(v) = v- 2B(e_i, v)e_i$, defines an automorphism of $V$.

\begin{theorem}[Tits, see {\cite[p. 96]{Bour4-6}}]
The map $\rho:W(M)\to GL(V)$ defined through $\rho(s_i) =\sigma_i$ is a faithful representation of $W(M)$.
\end{theorem}
The representation $\rho$ above is called the Tits (or geometrical) representation of $W(M)$. Let us denote by $I_m$ the identity matrix of size $m$.
In particular for $T_n$, the Tits representation $\rho$ is given by $t_i \mapsto \sigma_i$ where
$$
\sigma_1=\left(\begin{array}{cc|c}
        -1 & 0 & \raisebox{-10pt}{0}\\
        2 & 1 & \\ \hline
        \multicolumn{2}{c|}{0} & I_{n-3}\\
    \end{array}
    \right), \qquad
\sigma_{n-1}=\left(\begin{array}{c|cc}
        I_{n-3} & \multicolumn{2}{c}{0} \\ \hline
        \raisebox{-10pt}{0} & 1 & 2\\
         & 0 & -1\\
    \end{array}
    \right)
$$
and for $1 < i < n-1$, 
$$\sigma_i= \left(\begin{array}{c|ccc|c}
        I_{i-2} & \multicolumn{3}{c|}{0} & 0\\ \hline
         & 1 & 2 & 0 & \\
        0 & 0 & -1 & 0 & 0\\
         & 0 & 2 & 1 & \\ \hline
        0 & \multicolumn{3}{c|}{0} & I_{n-i-2}\\
    \end{array}
    \right).$$

\subsection{Chevyshev Polynomials}\label{subsec:Chebyshev}

There are two types of Chebyshev polynomials: the first and the second kind. For the purpose of this paper, we only need to recall the definition of the Chebyshev polynomials of second kind. For a more detailed treatment on Chebyshev polynomials see \cite{MasHan03}.

\begin{definition}
    The \textit{Chebyshev polynomial of second kind} $U_n(z)$ is the polynomial of degree $n$ in one variable $z$ defined by the recurrence relation
    \begin{equation}\label{eq:recurrence Chebyshev}
        U_n(z) = 2z U_{n-1}(z) - U_{n-2}(z),
    \end{equation}
    and the initial conditions
    \begin{equation}\label{eq:initial Chebyshev}
    U_0(z)=1 \quad \textrm{and} \quad U_1(z)=2z.
\end{equation}
\end{definition}

In the rest of the paper we simply refer to the polynomials $U_i$'s as the Chebyshev polynomials. In Section \ref{sec:Invariant}, we prove that certain polynomials associated to a family of twins can be written in terms of the Chebyshev polynomials (Corollary \ref{cor:chebyshev}).

\section{A Deformed Tits Representation for the Twin Group}\label{sec:Representation}

The main purpose of this section is to define a deformation of the Tits representation and to give some of its properties. 

For $n\geq 3$, we consider the following $(n-1)\times(n-1)$ matrices defined over the ring $\Lambda:=\Zxy$
$$
    V_1=\left(\begin{array}{cc|c}
            -1 & 0 & \raisebox{-10pt}{0}\\
            x & 1 & \\ \hline
             \multicolumn{2}{c|}{0} & I_{n-3}\\
        \end{array}
        \right), \qquad
    V_{n-1}=\left(\begin{array}{c|cc}
            I_{n-3} & \multicolumn{2}{c}{0} \\ \hline
            \raisebox{-10pt}{0} & 1 & y\\
             & 0 & -1\\
        \end{array}
        \right)
$$
and for $1 < i < n-1$, 
$$V_i= \left(\begin{array}{c|ccc|c}
        I_{i-2} & \multicolumn{3}{c|}{0} & 0\\ \hline
         & 1 & y & 0 & \\
        0 & 0 & -1 & 0 & 0\\
         & 0 & x & 1 & \\ \hline
        0 & \multicolumn{3}{c|}{0} & I_{n-i-2}\\
    \end{array}
    \right).$$
For $n=2$, we set $V_1 := -1$.

\begin{lemma}\label{lem:braid rel estimation}
    For any $n\geq 3$ and $1\leq i \leq n-2$,
    \begin{equation*}
        V_{i}V_{i+1}V_{i}-V_{i+1}V_iV_{i+1} = (xy-1)(V_i-V_{i+1}).
    \end{equation*}
\end{lemma}

\begin{proof}
    It suffices to verify the case $i=2$ and $n=4$, which follows from a direct computation. 
\end{proof}

Matrices $V_i$'s are an analogue to the matrices used in the reduced Burau representation of the braid group \cite{Bur35}. In fact, we have the following proposition.

\begin{proposition}\label{prop:psi is rep}
    For $n\geq 2$, the function $\psi_n:\B n \longrightarrow \gl{n-1}$ defined by $\psi_n(\gen{i})=V_i$ is a representation of the twin group.
\end{proposition}

\begin{proof}
  The proof follows by checking the defining relations of the twin group are satisfied by the matrices $V_i$'s. For $n=3$ the checking results from a direct computation. In the general case, it is enough to check for $n=4$, which follows again by direct computation:
  \begin{equation*}
      (\psi_4(\gen{2}))^2=\left(\begin{array}{ccc}
         1 & y & 0 \\
         0 & -1 & 0\\
         0 & x & 1 
    \end{array}
    \right) \left(\begin{array}{ccc}
         1 & y & 0 \\
         0 & -1 & 0\\
         0 & x & 1 
    \end{array} 
    \right) = \left(\begin{array}{ccc}
         1 & 0 & 0 \\
         0 & 1 & 0\\
         0 & 0 & 1 
    \end{array}
    \right)
  \end{equation*}
  and
  \begin{eqnarray*}
    \psi_4(\gen{1})\psi_4(\gen{3}) &=& \left(\begin{array}{ccc}
         -1 & 0 & 0 \\
         x & 1 & 0 \\
         0 & 0 & 1 \\
    \end{array}
    \right) \left(\begin{array}{ccc}
         1 & 0 & 0\\
         0 & 1 & y\\
         0 & 0 & -1
    \end{array} 
    \right) = \left(\begin{array}{ccc}
        -1 & 0 & 0\\
         x & 1 & y\\
         0 & 0 & -1
    \end{array}
    \right) \\  &=& \left(\begin{array}{ccc}
         1 & 0 & 0\\
         0 & 1 & y\\
         0 & 0 & -1
    \end{array}
    \right) 
    \left(\begin{array}{ccc}
         -1 & 0 & 0 \\
         x & 1 & 0 \\
         0 & 0 & 1 \\
    \end{array}
    \right) = \psi_4(\gen{3})\psi_4(\gen{1}).
  \end{eqnarray*}
\end{proof}

\begin{remark}
    Specializing $x=y=2$, the matrices $V_i$'s become the matrices $\sigma_i$'s seen in Subsection \ref{subsec:Coxeter}. Thus, $\psi_n$ can be regarded as a two parameters deformation of the Tits representation of the twin group $\B n$.
\end{remark}

Observe that,
\begin{equation}\label{eq:Vi *i}
    \psi_{n+1}(\iota^R(\gen {i}))= \left(
    \begin{array}{cc}
    \psi_n(\gen {i}) & 0\\
    v^R(\gen i) & 1\\ 
    \end{array}
    \right) = \left(
    \begin{array}{cc}
    V_i & 0\\
    v^R(\gen i) & 1\\
    \end{array} \right) \qquad 1\leq i \leq n-1,
\end{equation}
where $v^R(\gen i)$ is the row of length $n-1$ equal to $0$ if $i<n-1$ and $(0,\dots,0,x)$ if $i=n-1$. This implies that for $n\geq 2$ and $\beta \in \B n$, 
\begin{equation}\label{eq:matrix *beta}
    \psi_{n+1}(\iota^R(\beta))= \left(
    \begin{array}{cc}
    \psi_n(\beta) & 0\\
    v^R(\beta) & 1\\
    \end{array}
    \right),
\end{equation}
where $v^R(\beta)$ is a row of length $n-1$ over $\Lambda$ depending on $\beta$. In Corollary \ref{cor:row *b} we will see how to compute $v^R(\beta)$.

\begin{remark}\label{rem:b[1]}
    The analogous of Equation (\ref{eq:matrix *beta}) for the inclusion $\iota^L$ is
    \begin{equation}\label{eq:matrix *beta[1]}
        \psi_{n+1}(\iota^L(\beta))= \left(
        \begin{array}{cc}
        1 & v^L(\beta)\\
        0 & \psi_n(\beta)\\
        \end{array}
        \right),
    \end{equation}
    where $v^L(\beta)$ is a row of length $n-1$ over $\Lambda$ depending on $\beta$.
\end{remark}

In order to compute the values of $\psi_n$ on the elements defined in \ref{eq: notation t_n,i} and \ref{eq: notation t_1,i}, we need to introduce the following notation: given a matrix $M = (m_{r,s})$ of size $k \times l$, we denote $M^{\tau} = (m'_{r,s})$ the matrix whose entries are $m'_{r,s} := m_{k-r+1,l-s+1}$. For $n\geq 2$, let $A_n = (a_{i,j})$, $\tilde{Z}_n = (\tilde{z}_{i,j})$ be the $n\times n$ matrices given by
\[
\scalemath{0.95}{
a_{i,j} := \begin{cases}
                -y^{j-1}  & i=1,\, 1\leq j\leq n,\\
                0         & i<j-1,\\
                x       &  i=j+1,\, 1\leq j \leq n-1,\\
                y^{j-i}(xy-1) & 1< j\leq i,    
            \end{cases} \quad 
\tilde{z}_{i,j} := \begin{cases}
                xy^j -2y^{j-1}         & i=1,\\
                x^iy^j -x^{i-1}y^{j-1}  & 1< i< n,\\
                -x^{n-1}y^{j-1}       & i= n.
            \end{cases}
}        
\]
Finally, we set $Z_n:=\tilde{Z}_n+I_n$. For instance, for $n=3$ we have
\[
\scalemath{0.95}{
    A_3=\left(\begin{array}{rrrr}
        -1 & -y & -y^{2}  \\
        x & x y - 1 & x y^{2} - y \\
        0 & x & x y - 1 \\
    \end{array}\right), \qquad  
    Z_3=\left(\begin{array}{rrr}
        x y - 1 & x y^{2} - 2 y & x y^{3} - 2 y^{2} \\
        x^{2} y - x & x^{2} y^{2} - x y + 1 & x^{2} y^{3} - x y^{2} \\
        -x^{2} & -x^{2} y & -x^{2} y^{2} + 1
    \end{array}\right).
}
\]

Let $M_{k,l}(\Lambda)$ denote the ring of $k\times l$ matrices over $\Lambda$. In the lemma below we use the natural map $\theta:M_{k,l}(\Lambda)\to M_{k,l}(\Lambda)$ induced from the $\mathbb{Z}$-automorphism of $\Lambda$ that interchanges  $x$ with $y$.

\begin{lemma}\label{lem:matrices}
    For any $n\geq 2$ and $1\leq i\leq n-1$,
    \begin{enumerate}[label=(\alph{enumi})]
        \item $\psi_{n+1}(\gen 1\cdots\gen {n})= A_n$, \label{eq: matriz A}
        \item $\psi_{n+1}(\gen {n,i})= B_{n,i}$, \label{eq: matriz Bi}
        \item $\psi_{n+1}(\gen {1,i})= C_{n,i}$, \label{eq: matriz Cj}
    \end{enumerate} 
    where $B_{n,i}$ and $C_{n,i}$ are defined as follows:
    \begin{equation*}
    \scalemath{0.88}{
    B_{n,i}=\left\{
        \begin{array}{cl}
            \left( \begin{array}{c|c}
            I_{n-i-1}   & Y_{i+1}\\ \hline
                0       & Z_{i+1}
           \end{array} \right)  & \textrm{$1\leq i\leq n-2$},\\
                                &\\
            Z_{n}             & \textrm{$i=n-1$},
     \end{array}
     \right. \quad C_{n,i}=\left\{
        \begin{array}{cl}
            \left( \begin{array}{c|c}
                \theta (Z_{i+1})^\tau  & 0 \\ \hline
                \theta(Y_{i+1})^\tau  & I_{n-i-1}
           \end{array} \right)  & \textrm{$1\leq i\leq n-2$},\\
                                &\\
            \theta(Z_{n})^\tau        & \textrm{$i=n-1$},
     \end{array}
     \right.}
    \end{equation*}
   and $Y_i$ is the $(n-i-1) \times i$ matrix given by
    \begin{equation*}
    \scalemath{0.88}{
     Y_{i}=\left(
          \begin{array}{c}
            \vdots \\ \hline
            \begin{array}{cccc}
                y & y^2 & \cdots  & y^{i}
            \end{array}
          \end{array}
         \right),
        }
    \end{equation*}
    where vertical dots mean a block of 0's for $1\leq i<n-2$.
\end{lemma}

\begin{proof}
 \begin{itemize}
     \item[(a)] It follows by induction on $n$ and Equation (\ref{eq:matrix *beta}).
     \item[(b)] For $1\leq i\leq n-2$, we proceed by induction over $i$. For $i=1$ the result is straightforward. We want to prove $\psi_{n+1}(\gen {n,i+1})=B_{n,i+1}$. Using Lemma~\ref{lem:braid rel estimation}, we obtain
    \begin{eqnarray}\label{eq:tn,i+1}
        \psi_{n+1}(\gen {n,i+1})&=&\psi_{n+1}(\gen {n-i-1})\psi_{n+1}(\gen {n,i})\psi_{n+1}(\gen {n-i-1}) \nonumber \\
        && \!\! -(xy-1)(\psi_{n+1}(\gen {n-i-1})-\psi_{n+1}(\gen {n,i})).
    \end{eqnarray}
    Since $\psi_{n+1}(\gen {n,i})=B_{n,i}$ by induction hypothesis, we have
        $$\psi_{n+1}(\gen {n,i})=\left(
            \begin{array}{c|ccc|c}
                I_{n-i-3}   &  &  &  &   \\ \hline
                &1  & 0 & 0 & \\
                & 0 & 1 & y & r_1 \\
                 & 0 & 0 & xy-1 & r_2 \\ \hline  
                &  &  & c_1 & A
            \end{array}
              \right),
        $$
    where $r_1$ and $r_2$ are row vectors of length $i$, $c_1$ is a column vector of length $i$ and $A$ is a $i\times i$ matrix. Hence
    $$\scalemath{0.9}{\psi_{n+1}(\gen {n-i-1})\psi_{n+1}(\gen {n,i})\psi_{n+1}(\gen {n-i-1})=\left(
            \begin{array}{c|ccc|c}
                I_{n-i-3}   &  &  &  &   \\ \hline
                &1 & x y^{2} & y^{2}  & yr_1  \\
                &0 & -x y + 1 & -y & -r_1 \\
                &0 & 2 x^{2} y - 2 x & 2 x y - 1 & xr_1+r_2 \\ \hline
                & & xc_1 & c_1 & A
            \end{array}
          \right)
    }
    $$
    and
    $$\psi_{n+1}(\gen {n-i-1})-\psi_{n+1}(\gen {n,i})=\left(
            \begin{array}{c|ccc|c}
                 &  &  &  &   \\ \hline       
                &0  &y & 0 &  \\
                & 0 & -2 & -y & -r_1 \\
                & 0& x & 2-xy & -r_2 \\ \hline
                & & & -c_1 & I_i-A
            \end{array}
          \right).$$      
    Replacing the expressions above in Equation (\ref{eq:tn,i+1}), we obtain 
    $$\psi_{n+1}(\gen {n,i+1})=\left(
            \begin{array}{c|ccc|c}
                I_{n-i-3}   &  &  &  &   \\ \hline
                &1  & y & y^2 & yr_1  \\
                & 0 & xy-1 & y(xy-2) & (xy-2)r_1 \\
                & 0& x^2y-x & y(x^2y-x)+1 & xr_1+xyr_2 \\ \hline
                &  & xc_1 & xyc_1 & (1-xy)I_i-xyA
            \end{array}
          \right).$$
    Finally, using that $xyr_2=(x^2y-2x)r_1$, we obtain that $\psi_{n+1}(\gen {n,i+1})=B_{n,i+1}$. For the case $i=n-1$, the result follows by using Lemma \ref{lem:braid rel estimation} and the formula for $B_{n,n-2}$.
     \item[(c)] This case follows by Lemma \ref{lem:braid rel estimation} and a similar calculation as in (b).
 \end{itemize} 
\end{proof}

From now on, we take $x=y$. Thus, the representation $\psi_n$ is defined over $\Lambda=\Zx$.
For $n\geq 2$ define the vectors $F^R$ and $F^L$ of $\Lambda^n$ as follows:
$$F^R:=(\U{0},\U{1},\dots,\U{n-2},\U{n-1}) \in \Lambda^n$$ 
and 
$$F^L:=(\U{n-1},\U{n-2},\dots,\U{1},\U{0}) \in \Lambda^n,$$
where $U_i(z)$ is the $i$-th Chebyshev polynomial.

\begin{proposition}\label{prop:vector F}
    For any $\beta \in \B {n}$, the vectors $F^R$ and $F^L$ satisfy the following equations:
    \begin{enumerate}[label=(\alph{enumi})]
        \item  $F^R \psi_{n+1}(\iota^R(\beta))=F^R$,
        \item  $F^L \psi_{n+1}(\iota^L(\beta))=F^L$.
    \end{enumerate}
\end{proposition}

\begin{proof}
    (a) It suffices to prove it for matrices of the form (\ref{eq:Vi *i}). Expanding equations
    \begin{equation*}
        F^R \left(
        \begin{array}{cc}
            V_i & 0\\
            v^R(\gen i) & 1\\
        \end{array} \right) = F^R \qquad \textrm{for $0\leq i \leq n-1$ },
    \end{equation*}
    it turns out to be equivalent to the system of equations:
    \begin{eqnarray}
        -2\U{0} + x\U{1} &=& 0, \label{eq:f_0 f_1} \\
        x\U{i-1} - 2\U{i} + x\U{i+1} &=& 0. \qquad (2\leq i \leq n-2) \label{eq:f_i recurrence}
    \end{eqnarray}
    Now, Equation (\ref{eq:f_0 f_1}) is satisfied by the initial conditions (\ref{eq:initial Chebyshev}) of the Chebyshev polynomials, and Equation (\ref{eq:f_i recurrence}) is exactly their recurrence relation (\ref{eq:recurrence Chebyshev}). So, the proof is concluded.
    
    (b) This case is treated similarly.
\end{proof}

The next corollary allow us to compute vectors $v^R(\beta)$, $v^L(\beta)$ on matrices (\ref{eq:matrix *beta}), (\ref{eq:matrix *beta[1]}), respectively. This will be useful to understand the compatibility of the representation $\psi_n$ with the inclusions $\iota^R$ and $\iota^L$, essential on the Gotin-Markov moves \ref{rel:M0}, \ref{rel:M2}, \ref{rel:M3}.

\begin{corollary}\label{cor:row *b}
    Let $\beta \in \B n$ and let $(a_{1,k} \, a_{2,k}  \cdots  a_{n-1,k})$ be the $k$-th column of the matrix $\psi_n(\beta)-I_{n-1}$ for $1\leq k\leq n-1$. Then, the entries of $v^R(\beta)=(b_1 \, b_2  \cdots  b_{n-1})$ in matrix (\ref{eq:matrix *beta}) and $v^L(\beta)=(c_1 \, c_2  \cdots  c_{n-1})$ in matrix (\ref{eq:matrix *beta[1]}) satisfy:
    \begin{enumerate}[label=(\alph{enumi})]
        \item $-\U{n-1}b_k = \sum_{i=1}^{n-1} \U{i-1}a_{i,k},$
        \item $-\U{n-1}c_k = \sum_{i=1}^{n-1} \U{n-2-i}a_{i,k}.$
    \end{enumerate}
\end{corollary}

\begin{proof}
    We only prove (a), a similar argument works for (b). By Proposition \ref{prop:vector F}, 
    \[
        F^R \left(
        \begin{array}{cc}
        \psi_n(\beta) & 0\\
        v^R(\beta) & 1\\
        \end{array}
        \right) = F^R.
    \]
    Because $F^R I_n=F^R$, we obtain
    \begin{equation}\label{eq: FM=0}
        F^R \left(
        \begin{array}{cc}
        \psi_n(\beta) - I_{n-1} & 0\\
        v^R(\beta) & 0\\
        \end{array}
        \right) = 0.
    \end{equation}
    Writing
    \begin{equation*}
        \left(
        \begin{array}{c}
            \psi_n(\beta) - I_{n-1}\\
            v^R(\beta) \\
        \end{array}\right) 
        = \left(\begin{array}{cccc}
            a_{11} & a_{12} & \cdots & a_{1,n-1} \\
            a_{21} & a_{22} & \cdots & a_{2,n-1} \\
            \vdots & \vdots & \ddots & \vdots \\
            a_{n-1,1} & a_{n-1,2} & \cdots & a_{n-1,n-1} \\
            b_1 & b_2 & \cdots & b_{n-1}\\
        \end{array}\right),
    \end{equation*}
    and expanding the product of Equation (\ref{eq: FM=0}), we obtain the equations
    $$a_{1,k} + \U{1} a_{2,k} + \U{2} a_{3,k} + \cdots + \U{n-2} a_{n-1,k} + \U{n-1} b_k = 0,$$
    where $1\leq k \leq n-1$. These  equations imply the claim (a).
\end{proof}

\section{An Invariant of Doodles}\label{sec:Invariant}

This is the main section of the paper and consists in two subsections. In the first one we introduce a family of functions $f_n$'s and we study its behavior under the Gotin-Markov moves. Also, we give a a skein relation that is satisfied by $f_n$. In the next subsection we present our polynomial invariant for doodles $\mathcal{Q}$.

\subsection{The function $f_n$}\label{subsec:f_n}

\begin{definition}\label{def:pol Pn}
    For all $n\geq 1$, we define the polynomial $P_{n}(x)\in \Lambda$ by the formula
    \begin{equation*}
        P_{n}(x)=\det(\psi_{n+1}(\gen 1\dots\gen {n})-I_{n}).
    \end{equation*}
    We set $P_0(x)=1$.
\end{definition}

\begin{proposition}\label{eq:Pn recurrencia}
    For $n\geq 2$, the polynomials $P_n$ satisfy the following recurrence relation:
    \begin{equation*}
        P_{n}(x)=-2P_{n-1}(x)-x^2P_{n-2}(x).  
    \end{equation*}
\end{proposition}

\begin{proof}
    By Lemma~\ref{lem:matrices}, $P_n(x)=\det(A_n-I_n)$, where
        $$A_n-I_n= \left(
        \begin{array}{cccccc}
            -2 & -x & -x^2 &\cdots& -x^{n-2}& -x^{n-1} \\
            x & x^2-2  & x(x^2-1)  & \cdots&x^{n-3}(x^2-1) & x^{n-2}(x^2-1)  \\
            0 & x & x^2-2 &\cdots& x^{n-4}(x^2-1)  & x^{n-3}(x^2-1)  \\
            \vdots & \vdots &\ddots  &\ddots&\vdots & \vdots\\
            0 & 0 & \cdots   & x& x^2-2 & x(x^2-1) \\
            0 & 0 & \cdots   & 0& x & x^2-2
        \end{array}
        \right).$$
    Then, expanding the determinant with respect the last row, we get
    \begin{equation}\label{eq1}
        P_n(x)= (x^2-2)P_{n-1}(x)-x\det(M),
    \end{equation}
    where $M$ is the following $(n-1)\times (n-1)$ matrix
    
    $$M=\left(
      \begin{array}{cccccc}
        -2 & -x & -x^2 &\cdots& -x^{n-3}& -x^{n-1} \\
        x & x^2-2  & x(x^2-1)  & \cdots&x^{n-4}(x^2-1) & x^{n-2}(x^2-1)  \\
        0 & x & x^2-2 &\cdots& x^{n-5}(x^2-1)  & x^{n-3}(x^2-1)  \\
        \vdots & \vdots &\ddots  &\ddots&\vdots & \vdots\\
        0 & 0 & \cdots  & x& x^2-2 & x^2(x^2-1) \\
        0 & 0 & \cdots  & 0& x & x(x^2-1)
      \end{array}
    \right).$$ 
    We have that $\det(M)=x\det(M')$, where 
    $$M'=\left(
      \begin{array}{cccccc}
        -2 & -x & -x^2 &\cdots& -x^{n-3}& -x^{n-2} \\
        x & x^2-2  & x(x^2-1)  & \cdots&x^{n-4}(x^2-1) & x^{n-3}(x^2-1)  \\
        0 & x & x^2-2 &\cdots& x^{n-5}(x^2-1)  & x^{n-4}(x^2-1)  \\
        \vdots & \vdots &\ddots  &\ddots&\vdots & \vdots\\
        0 & 0 & \cdots  & x& x^2-2 & x(x^2-1) \\
        0 & 0 & \cdots  & 0& x & x^2-1
      \end{array}
    \right).$$ 
    Finally, it is not difficult to prove that $\det(M')=P_{n-1}(x)+P_{n-2}(x)$. Then, the result follows by replacing it in Equation (\ref{eq1}).
\end{proof}

\begin{remark}
    Observe that by Proposition (\ref{eq:Pn recurrencia}), $P_n(x)\in \mathbb{Z}[x^2]$ for all $n$.
\end{remark}

\begin{corollary}\label{cor:chebyshev}
  For any $n\geq 0$,
  $$P_n(x)=(-x)^{n}\U{n}.$$ 
\end{corollary}
\begin{proof}
  Firstly, we have that $P_0(x)=1=(-x)^0\U{0}$ and $P_1(x)=-2=-x\U{1}$. Secondly, by the recurrence equation of the Chebyshev polynomials (\ref{eq:recurrence Chebyshev}), we have
  \begin{equation*}
  \U{n}=\frac{2}{x}\U{n-1}-\U{n-2},
  \end{equation*}
  which implies $(-x)^n\U{n}=-2(-x)^{n-1}\U{n-1}-x^2(-x)^{n-2}\U{n-2}$. That is, the polynomial $(-x)^n\U{n}$ satisfies the same defining rules of $P_n(x)$, hence they are all equal.
\end{proof}

\begin{definition}
    Let $f_n:\B n\to \Lambda$ be the map defined by
    \begin{equation*}
        f_n(\beta):= \left\{ \begin{array}{cc}
             1 & \textrm{ if } n=1, \\
            & \\         
            \frac{\det(\psi_{n}(\beta)-I_{n-1})}{(-x)^{n-1}\U{n-1}} & \textrm{ if } n\geq 2.
        \end{array}\right.
    \end{equation*}  
    In particular $f_n(1)=0$ and $f_n(\gen 1\gen 2\dots\gen {n-1})=1$ for any $n\geq 2$.
\end{definition}

The rest of the subsection is devoted to investigating the compatibility of functions $f_n$'s under the Gotin-Markov moves, it is divided into Lemmas \ref{lem:fn M0 y M1}, \ref{lem:fn simple stab}, \ref{lem:hyper-stabil} and Proposition \ref{prop:skein fn}. Some of the proofs might be troublesome by their length, nevertheless the calculations are straightforward and we show the required steps. In the beginning of Subsection \ref{subsec:invariant Q}, it is summarized the properties proved in the rest of this subsection.

\begin{lemma}\label{lem:fn M0 y M1}
    For any $n\geq 2$ and $\alpha,\beta \in \B {n}$,
    \begin{enumerate}[label=(\alph*)]
        \item  $f_{n+1}(\iota^R(\beta)) = f_{n+1}(\iota^L(\beta))=0,$ \label{itm prop:M0 fn}
        \item  $f_n(\alpha^{-1}\beta\alpha) = f_n(\beta).$ \label{itm prop:comm fn}
    \end{enumerate}
\end{lemma}

\begin{proof}
    The proof of \ref{itm prop:M0 fn} follows by Equation (\ref{eq:matrix *beta}) and Equation (\ref{eq:matrix *beta[1]}). The proof of \ref{itm prop:comm fn} follows from conjugacy properties of determinants.
\end{proof}

\begin{notation}
    In subsequent proofs, $C_{i,j}(k)$ (resp. $R_{i,j}(k)$) denotes the elementary operation which replaces the $i$-th column $C_i$ (resp. row $R_i$) by the sum $C_i + kC_j$ (resp. $R_i + kR_j$). For $k\neq 0$, $C_i(k)$ (resp. $R_i(k)$) denotes the elementary operation which multiplies by $k$ the $i$-th column (resp. row), and $C_{i,j}$ (resp. $R_{i,j}$) denotes the elementary operation that exchanges $i$-th and $j$-th columns (resp. rows).
\end{notation}

\begin{lemma}\label{lem:fn simple stab}
    For any $n\geq 2$ and $\alpha,\beta \in \B {n}$,
    \begin{enumerate}[label=(\alph*)]
        \item  $f_{n+1}(\iota^R(\beta) \gen n) = f_n(\beta)$, \label{itm prop:stabil1 fn}
        \item  $f_{n+1}(\iota^L(\beta) \gen 1) = f_n(\beta)$. \label{itm prop:stabil2 fn}
    \end{enumerate}
\end{lemma}
    
\begin{proof}
    \ref{itm prop:stabil1 fn} The equation $f_{n+1}(\iota^R(\beta) \gen n) = f_n(\beta)$ is equivalent to
    \begin{equation}\label{eq:equiv eq b = bs_n}
        \U{n} \det(\psi_{n+1}(\iota^R(\beta)\gen n)-I_{n}) = -x\U{n}\det(\psi_{n}(\beta)-I_{n-1}).
    \end{equation}
    Let 
    \begin{equation}\label{eq:matrix in blocks ABCD...}
        \left(
            \begin{array}{cc}
            \psi_n(\beta) & 0\\
            v^R(\beta) & 1\\
            \end{array}
        \right) = 
        \left(
            \begin{array}{ccc}
                A & B & 0\\
                C & D & 0\\
                E & F & 1\\
            \end{array}
        \right), \qquad \textrm{(cf. \cite[Lemma~3.12]{KasTur08})}
    \end{equation}
    where $A$ is a square matrix of size $(n-2)$, $B$ is a column of height $n-2$, $C$ and $E$ are rows of length $n-2$, and $D,F \in \Lambda$. By Equation (\ref{eq:matrix *beta}),
    \begin{align}
            \psi_{n+1}(\iota^R(\beta)\gen n) &= \left(
                \begin{array}{cc}
                \psi_n(\beta) & 0\\
                v^R(\beta) & 1\\
                \end{array}
            \right)
            \left(
                \begin{array}{c|c}
                I_{n-2} & 0\\
                \hline
                0 & \begin{array}{cc}
                    1 & x\\
                    0 & -1\\
                    \end{array}\\
                \end{array}
            \right) \nonumber \\
        &= \left(
                \begin{array}{ccc}
                    A & B & 0\\
                    C & D & 0\\
                    E & F & 1\\
                \end{array}
            \right)
            \left(
                \begin{array}{c|c}
                I_{n-2} & 0\\
                \hline
                0 & \begin{array}{cc}
                    1 & x\\
                    0 & -1\\
                    \end{array}\\
                \end{array}
            \right) \nonumber \\
        &= \left(
                \begin{array}{ccc}
                    A & B & xB\\
                    C & D & xD\\
                    E & F & xF-1\\
                \end{array}
            \right). \label{eq:matrix beta sigma_n}
    \end{align}
    Subtracting the identity to (\ref{eq:matrix beta sigma_n}), we have
    \begin{align}
            \psi_{n+1}(\iota^R(\beta)\gen n) - I_{n}&= \left(
                \begin{array}{ccc}
                    A-I_{n-2} & B & xB\\
                    C & D-1 & xD\\
                    E & F & xF-2\\
                \end{array}
            \right) \xrightarrow{C_{n,n-1} (-x)}\cdots \nonumber \\
             \cdots & \to \left(
                \begin{array}{ccc}
                    A-I_{n-2} & B & 0\\
                    C & D-1 & x\\
                    E & F & -2\\
                \end{array}
            \right) = M.
    \end{align}
    Hence, $\det(\psi_{n+1}(\iota^R(\beta)\gen n) - I_{n})=\det(M).$  Observe that
    \begin{equation*}
    \left(
        \begin{array}{c}
            \psi_n(\beta)-I_{n-1} \\
            v^R(\beta) \\
        \end{array}
    \right) =
    \left(
        \begin{array}{ccc}
            A-I_{n-2} & B\\
            C & D-1 \\
            E & F \\
        \end{array}
    \right). 
    \end{equation*}
    Then, 
    \begin{align*}
    \U{n-1}\det(\psi_{n+1}(\iota^R(\beta)\gen n) - I_{n}) &= \U{n-1} \det (M)\\
     &= \det \left(
            \begin{array}{ccc}
                A-I_{n-2} & B & 0\\
                C & D-1 & x\\
                \U{n-1}E & \U{n-1}F & -2\U{n-1}\\
            \end{array}
        \right)\\
        &= \det(N).
    \end{align*}
    For $1\leq i \leq n-2$, if we add $\U{i-1}$ times the $i$-th row of $N$ to its $n$-th row, by Lemma \ref{cor:row *b} and the recurrence relation of Chebyshev polynomials (\ref{eq:recurrence Chebyshev}), we have
    \begin{align*}
    \det (N) &= \det \left(
            \begin{array}{ccc}
                A-I_{n-2} & B & 0\\
                C & D-1 & x\\
                0 & 0 & -2\U{n-1} + x\U{n-2}\\
            \end{array}
        \right)\\
        &= (-2\U{n-1} + x\U{n-2})\det(\psi_n(\beta)-I_{n-1})\\
        &= -x\U{n}\det(\psi_n(\beta)-I_{n-1}).
    \end{align*}
    Obtaining the Equation (\ref{eq:equiv eq b = bs_n}) as desired.
    
    \ref{itm prop:stabil2 fn} This case is treated similarly by Remark \ref{rem:b[1]}.
\end{proof}

\newpage

\begin{lemma}\label{lem:hyper-stabil}
    For any $n\geq 3$, $\beta \in \B {n}$ and $1\leq i\leq n-1$,
    \begin{enumerate}[label=(\alph*)]
          \item  $f_{n+1}(\iota^R(\beta)\gen {n,i}) = x^{2i}f_n(\beta)$, \label{itm prop:hyper-stab1 fn}
          \item  $f_{n+1}(\iota^L(\beta)\gen {1,i}) = x^{2i}f_n(\beta)$. \label{itm prop:hyper-stab2 fn}
      \end{enumerate}
\end{lemma}

\begin{proof}
    \ref{itm prop:hyper-stab1 fn} We will prove it for $1\leq i\leq n-2$. The proof of the case $i=n-1$ is treated similarly with the corresponding changes for the matrix $\psi_{n+1}(\gen{n,n-1})$ (see Lemma \ref{lem:matrices}).
    
    The equation $f_{n+1}(\iota^R(\beta)\gen {n,i}) = x^{2i}f_n(\beta)$ is equivalent to
    \begin{equation}\label{eq:equivalent to fn+1 = x^2if_n}
        \U{n-1}\det(\psi_{n+1}(\iota^R(\beta)\gen{n,i}-I_n)) = -x^{2i+1}\U{n}\det(\psi_n(\beta-I_{n-1})).
     \end{equation}
     As in Equation (\ref{eq:matrix in blocks ABCD...}) of the previous lemma, we express $\psi_{n+1}(\iota^R(\beta))$ in blocks as
     \begin{equation}\label{eq:matrix i(b) many blocks}
     \left(
            \begin{array}{cc}
            \psi_n(\beta) & 0\\
            v^R(\beta) & 1\\
            \end{array}
        \right) =
        \left(
            \begin{array}{ccc}
                A & B & 0\\
                C & D & 0\\
                E & F & 1
            \end{array}
        \right),
     \end{equation}
     where $A$ is a square matrix of size $(n-i-2)$, $C$ is a $(i+1)\times (n-i-2)$ matrix, $D$ is a square matrix of size $(i+1)$, $B$ and $F$ are rows of length $i+1$ and $E$ is a row of length $(n-i-2)$. Note that for the case $i=n-2$, we omit the blocks $A$, $C$ and $E$. For every block, we write its entries by the same letter, for instance $B=(B_l)_{1\leq l \leq i+1}$. Using Lemma \ref{lem:matrices}, we know the description of $\psi_{n+1}(\gen {n,i})$, obtaining the product
     \begin{equation*}
     \psi_{n+1}(\iota^R(\beta) \gen {n,i})= 
        \left(
            \begin{array}{cc}
                A & B' \\
                C & D' \\
                E & F' 
            \end{array}
        \right),
    \end{equation*}
    where $A$, $C$, $E$ are as in (\ref{eq:matrix i(b) many blocks}), $B'$ and $F'$ are rows of length $i+2$, and $D'$ is a $(i+1)\times (i+2)$ matrix. By direct computation, we obtain that the entries of the blocks are the following:
    \begin{align*}
        B'_{l} &= \scalemath{0.85}{\begin{cases}
                     B_1 & l=1, \\
                     xB_{1} + (x^2-1)B_{2} + \sum\limits_{j=1}^{i-1} x^j(x^2-1)B_{j+2} & l=2, \\
                     x^{l-1}B_{1} + x^{l-2}(x^2-2)B_{2} + \sum\limits_{j=1}^{i-1} [x^{j+l-2}(x^2-1) + \delta_j^{l-2}] B_{j+2} & 3\leq l \leq i+1, \\
                     x^{i+1}B_{1} + x^{i}(x^2-2)B_{2} + \sum\limits_{j=1}^{i-1} x^{j+i}(x^2-1) B_{j+2} & l=i+2,
                \end{cases}
                }
                \\
        \end{align*}
        \begin{align*}
        D'_{k,l} &= \scalemath{0.85}{\begin{cases}
                     D_{k,1} & 1\leq k \leq i,\; l=1, \\
                     xD_{k,1} + (x^2-1)D_{k,2} + \sum\limits_{j=1}^{i-1} x^j(x^2-1)D_{k,j+2} & 1\leq k \leq i,\; l=2, \\
                     x^{l-1}D_{k,1} + x^{l-2}(x^2-2)D_{k,2} + \sum\limits_{j=1}^{i-1} [x^{j+l-2}(x^2-1) + \delta_j^{l-2}] D_{k,j+2} & 1\leq k \leq i, \; 3\leq l \leq i+1, \\
                     x^{i+1}D_{k,1} + x^{i}(x^2-2)D_{k,2} + \sum\limits_{j=1}^{i-1} x^{j+i}(x^2-1) D_{k,j+2} & 1\leq k \leq i, \; l=i+2,
                \end{cases}
                }
                \\
        F'_{l} &= \scalemath{0.85}{\begin{cases}
                     F_1 & l=1, \\
                     xF_{1} + (x^2-1)F_{2} + \sum\limits_{j=1}^{i-1} x^j(x^2-1)F_{j+2} - x^i & l=2, \\
                     x^{l-1}F_{1} + x^{l-2}(x^2-2)F_{2} + \sum\limits_{j=1}^{i-1} [x^{j+l-2}(x^2-1) + \delta_j^{l-2}] F_{j+2} - x^{i+l-2}& 3\leq l \leq i+1, \\
                     x^{i+1}F_{1} + x^{i}(x^2-2)F_{2} + \sum\limits_{j=1}^{i-1} x^{j+i}(x^2-1) F_{j+2} - x^{2i}+1 & l=i+2.
                \end{cases}
                }
    \end{align*}
    Subtracting the identity, we obtain the matrix $\psi_{n+1}(\iota^R(\beta)\gen {n,i}) - I_n$. Then, applying a sequence of elementary operations to this matrix, we obtain
    \[
        J = \scalemath{1.0}{
                \left(
                \begin{array}{ c | c}
                    & {0} \\
                    & x \\
                    \psi_{n}(\beta)-I_{n-1}& x^2-2 \\
                    & x(x^2-1) \\
                    & \vdots \\
                    & x^{i-1}(x^2-1) \\
                     \hline
                    v^R(\beta) & -x^i  \\
                \end{array}
                \right)
                }.
    \]
    More precisely, we have
    \begin{align*}
        & (\psi_{n+1}(\iota^R(\beta)\gen {n,i}) - I_n) \xrightarrow{\scalemath{0.7}{C_{n,n-i} (-x^{i})}}
        \bullet \xrightarrow{\scalemath{0.7}{R_{n} (-x^{-i})}}
        \bullet \xrightarrow[\scalemath{0.7}{\textrm{ (for $l=3,\dots,i+1$)}}]{\scalemath{0.7}{C_{n-i-2+l,n-i} (-x^{l-2})}}
        \bullet \xrightarrow[\scalemath{0.7}{\textrm{ (for $l=3,\dots,i+1$)}}]{\scalemath{0.7}{C_{n-i-2+l,n} (x^{l-2})}} \cdots \\
        & \cdots\to \bullet \xrightarrow{\scalemath{0.7}{C_{n-i,n-i-1}(-x)}}
        \bullet \xrightarrow{\scalemath{0.7}{C_{n-i,n} (-(x^2-1))}}
        \bullet \xrightarrow[\scalemath{0.7}{\textrm{ (for $j=1,\dots,i-1$)}}]{\scalemath{0.7}{C_{n-i,n-i+j} (-x^{j}(x^2-1)) }}
        \bullet \xrightarrow{\scalemath{0.7}{C_{n,n-i}}} J.
    \end{align*}
    By Corollary \ref{cor:row *b}, we know the description of $v^R(\beta)$ in terms of the rows of $\psi_{n}(\beta)-I_{n-1}$ and the Chebyshev polynomials. 
    Thus we obtain the matrix
    \begin{equation*}
            K = \scalemath{0.75}{
                \left(
                \begin{array}{ c | c}
                    & {0} \\
                    & x \\
                    \psi_{n}(\beta)-I_{n-1} & x^2-2 \\
                    & x(x^2-1) \\
                    & \vdots \\
                    & x^{i-1}(x^2-1) \\
                     \hline
                    0 & x\U{n-i-2} + (x^2-2)\U{n-i-1} + \sum\limits_{k=1}^{i-1}x^{i-k}(x^2-1)\U{n-1-k} - x^i\U{n-1} \\
                \end{array}
                \right)
                },
    \end{equation*}
    by applying the following sequence of elementary operations to $J$,
    \begin{eqnarray*}
            J & \xrightarrow{\scalemath{0.7}{R_{n}\left( \U{n-1} \right)}} \bullet \xrightarrow[\scalemath{0.7}{\textrm{ (for $j=1,\dots,n-1$)}}]{\scalemath{0.7}{R_{n,j} (\U{j-1})}} K.
    \end{eqnarray*}
    Using iteratively the recurrence relation of the Chebyshev polynomials (\ref{eq:recurrence Chebyshev}), we obtain
    \begin{equation*}
        \scalemath{0.825}{
        -x^{i+1}\U{n} = x\U{n-i-2} + (x^2-2)\U{n-i-1} + \sum\limits_{k=1}^{i-1}x^{i-k}(x^2-1)\U{n-1-k} - x^i\U{n-1}.
        }
    \end{equation*}
    Therefore,
    \begin{equation*}
            K = 
            \scalemath{1.0}{
                \left(
                \begin{array}{ c | c}
                    & {0} \\
                    & x \\
                    \psi_{n}(\beta)-I_{n-1} & x^2-2 \\
                    & x(x^2-1) \\
                    & \vdots \\
                    & x^{i-1}(x^2-1) \\
                     \hline
                    0 & -x^{i+1}\U{n} \\
                \end{array}
                \right).
                }
    \end{equation*}
    Finally, following the sequence of elementary operations that transform $\psi_{n+1}(\iota^R(\beta)\gen{n,i}-I_n)$ into $K$, we conclude that:
    \begin{eqnarray*}
        \det(\psi_{n+1}(\iota^R(\beta)\gen {n,i}) - I_n) &=& \frac{x^{i}}{\U{n-1}}\det(K) \\
        &=& -x^{2i+1}\frac{\U{n}}{\U{n-1}}\det(\psi_n(\beta)-I_{n-1}), 
    \end{eqnarray*}
    and Equation (\ref{eq:equivalent to fn+1 = x^2if_n}) follows.
    
    \ref{itm prop:hyper-stab2 fn} This case is treated similarly to \ref{itm prop:hyper-stab1 fn} by Remark \ref{rem:b[1]}. 
\end{proof}

\begin{example}\label{example:f(t1t2)^3}
    The twins $\alpha=(\gen{1}\gen{2})^3 \in \B{3}$ and $\beta=(\gen{1}\gen{2})^3\gen{3}\gen{2}\gen{3} \in \B{4}$ are Gotin-Markov equivalent. Let us compute their images under $f_n$. From the definition,
    \begin{equation*}
        \psi_3(\gen{1}) = \left(\begin{array}{cc}
            -1 & 0 \\
             x & 0 
        \end{array}\right)
        \quad \textrm{and} \quad
        \psi_3(\gen{2}) = \left(\begin{array}{cc}
             1 & x \\
             0 & -1
        \end{array}\right).
    \end{equation*}
    It follows that 
    \begin{equation*}
        \psi_3(\gen{1}\gen{2})^3 = \left(\begin{array}{cc}
            -x^4 + 3x^2 - 1 & -x^5 + 4x^3 - 3x \\
             x^5 - 4x^3 + 3x & x^6 - 5x^4 + 6x^2 - 1
        \end{array}\right).
    \end{equation*}
    Therefore
    \begin{align*}
        \det(\psi_3(\alpha)-I_2) &= \det\left(\begin{array}{cc}
            -x^4 + 3x^2 - 2 & -x^5 + 4x^3 - 3x \\
             x^5 - 4x^3 + 3x & x^6 - 5x^4 + 6x^2 - 2
        \end{array}\right)\\
        &=-x^6 + 6x^4 - 9x^2 + 4 \\
        &=(-x^2 + 4) (x^4 - 2x^2 + 1).
    \end{align*}
    Since $U_{2}(z)=-1+4z^2$, then
    \begin{align*}
        f_3(\alpha) &=\frac{\det(\psi_3(\alpha)-I_2)}{(-x)^2\U{2}}= \frac{(-x^2 + 4) (x^4 - 2x^2 + 1)}{-x^2 + 4}
        = x^4 - 2x^2 + 1.
    \end{align*}
    On the other hand,
    \begin{equation*}
        \psi_4(\gen{1}) = \left(\begin{array}{ccc}
            -1 & 0 & 0 \\
             x & 0 & 0 \\
             0 & 0 & 0
        \end{array}\right),\quad
        \psi_4(\gen{2}) = \left(\begin{array}{ccc}
             1 & x & 0 \\
             0 & -1 & 0 \\
             0 & x & 0
        \end{array}\right), \quad
        \psi_4(\gen{3}) = \left(\begin{array}{ccc}
             1 & 0 & 0 \\
             0 & 1 & x \\
             0 & 0 & -1
        \end{array}\right).
    \end{equation*}
    It follows that 
    \begin{equation*}
    \scalemath{0.85}{
        \psi_4((\gen{1}\gen{2})^3\gen{3}\gen{2}\gen{3}) = \left(\begin{array}{ccc}
            -x^4 + 3x^2 - 1 & -x^7 + 4x^5 - 4x^3 + 2x & -x^8 + 5x^6 - 8x^4 + 5x^2\\
            x^5 - 4x^3 + 3x & x^8 - 5x^6 + 7x^4 - 4x^2 + 1 & x^9 - 6x^7 + 12x^5 - 10x^3 + 2x\\
            x^4 - x^2 & x^7 - 2x^5 + 2x^3 - 2x & x^8 - 3x^6 + 4x^4 - 3x^2 + 1
        \end{array}\right).
        }
    \end{equation*}
    Therefore
    \begin{align*}
        \det(\psi_4(\beta)-I_3) &= 4x^8 - 16x^6 + 20x^4 - 8x^2
        =4 x^2 (x^2 - 2) (x^2 - 1)^2.
    \end{align*}
    Since $U_{3}(z)=-4z+8z^3$, then
    \begin{align*}
        f_4(\beta) &=\frac{\det(\psi_4(\beta)-I_3)}{(-x)^3\U{3}}= \frac{4 x^2 (x^2 - 2) (x^2 - 1)^2}{4x^2 - 8}
        = x^2 (x^2 - 1)^2 = x^6 - 2x^4 + x^2.
    \end{align*}
    Hence, $f_4(\beta)=x^{2}f_3(\alpha)$.
\end{example}

    Note that with Lemmas \ref{lem:fn M0 y M1}, \ref{lem:fn simple stab} and \ref{lem:hyper-stabil}, we have actually proved that if $\alpha\in \B {n}$ and $\beta \in \B {m}$ are Gotin-Markov equivalent, then $f_n(\alpha)$ and $f_m(\beta)$ are equal up to some factor $x^{2k}$ for some $k\in \mathbb{Z}$.

\begin{proposition}\label{prop:skein fn}
The function $f_n$ satisfies the skein relation:
  \begin{equation}\label{eq:skein rel}
        f_n\big( {\risSpdf{-9}{C121}{}{25}{0}{0}} \big) - f_n\big( {\risSpdf{-9}{C212}{}{25}{0}{0}} \big)
            = (x^2-1)\left(f_n\big( {\risSpdf{-9}{C1}{}{25}{0}{0}} \big) - f_n\big( {\risSpdf{-9}{C2}{}{25}{0}{0}} \big) \right).
    \end{equation}
\end{proposition}

\begin{proof}
    By \ref{itm prop:comm fn} of Lemma \ref{lem:fn M0 y M1}, the skein relation (\ref{eq:skein rel}) is equivalent to the equation
    \begin{equation*}
        f_n(\beta \gen {i}\gen {i+1}\gen {i}) - f_n(\beta \gen {i+1}\gen {i}\gen {i+1}) = (x^2-1)(f_n(\beta \gen {i}) - f_n(\beta \gen {i+1})),
    \end{equation*}
    for some $\beta \in \B n$. Therefore, (\ref{eq:skein rel}) is equivalent to prove the formula
    \begin{multline}\label{eq:equiv skein}
        \det(\psi_{n}(\beta\gen {i}\gen {i+1}\gen {i})-I_{n-1}) - \det(\psi_{n}(\beta\gen {i+1}\gen {i}\gen {i+1})-I_{n-1}) 
        \\= (x^2-1) [\det(\psi_{n}(\beta\gen {i})-I_{n-1}) - \det(\psi_{n}(\beta\gen {i+1})-I_{n-1})].
    \end{multline}
    The corresponding matrix for $\psi_n(\gen{i}\gen{i+1}\gen{i})$ is
    \begin{equation*}
        \psi_n(\gen{i}\gen{i+1}\gen{i}) = \left(
            \begin{array}{c | c  c  c  c | c }
                I_{i-2} & \multicolumn{4}{c|}{0} & 0\\
                 \hline
                 & 1 & x^3 & x^2 & 0 & \\
                 0 & 0 & -(x^2-1) & -x & 0 & 0 \\
                 & 0 & x(x^2-2) & x^2-1 & 0 & \\
                 & 0 & x^2 & x & 1 & \\
                 \hline
                 0  & \multicolumn{4}{c|}{0} & I_{n-i-3}\\
            \end{array}
        \right).
    \end{equation*}
    Then, we may write $\psi_n(\beta)$ in blocks as follows:
    \begin{eqnarray*}
         \psi_{n}(\beta) = 
            \left(\begin{array}{ccc}
                A & B & C \\
                D & E & F \\
                G & H & J \\
            \end{array}
            \right) &=&
             \left(\begin{array}{c | c c c c | c }
                        A & B_{1} & B_{2} & B_{3} & B_{4} & C\\
                         \hline
                         & \multicolumn{4}{c|}{\vdots} & \\
                         D & E_{k,1} & E_{k,2} & E_{k,3} & E_{k,4} & F \\
                         & \multicolumn{4}{c|}{\vdots} & \\
                         \hline
                         G & H_{1} & H_{2} & H_{3} & H_{4} & J
                    \end{array}
            \right), \qquad (1\leq k \leq 4) \nonumber
         \end{eqnarray*}
    where $A$ is square matrix of size $(i-2)$, $B$ is a $(i-2)\times(4)$ matrix, $C$ is a $(i-2)\times(n-i-3)$ matrix, $D$ is a $4\times(i-2)$ matrix, $E$ is a square matrix of size $4$, $F$ is a $4\times(n-i-3)$ matrix, $G$ is a $(n-i-3)\times(i-2)$ matrix, $H$ is a $(n-i-3)\times 4$ matrix and $J$ is a square matrix of size $(n-i-3)$. For $1\leq l\leq 4$, we write $B_l$ and $H_l$ for the columns of $B$ and $H$, respectively. The resulting product of $\psi_n(\beta)$ and $\psi_n(\gen{i}\gen{i+1}\gen{i})$ is:
    \begin{equation*}
    \scalemath{0.7}{
        \left(
            \begin{array}{c | c | c | c | c | c }
                A & B_{1}  
                    & x^3B_{1} - (x^2-1)B_{2} + x(x^2-2)B_{3} + x^2B_{4} 
                    & x^2B_{1} - xB_{2} + (x^2-1)B_{3} + xB_{4}
                    & B_{4}
                    & C \\
                \hline
                & \vdots & \vdots & \vdots & \vdots & \\
                D   & E_{k,1}  
                    & x^3E_{k,1} - (x^2-1)E_{k,2} + x(x^2-2)E_{k,3} + x^2E_{k,4} 
                    & x^2E_{k,1} - xE_{k,2} + (x^2-1)E_{k,3} + xE_{k,4}
                    & E_{k,4}
                    & F \\
                & \vdots & \vdots & \vdots & \vdots & \\
                \hline
                G   & H_{1}  
                    & x^3H_{1} - (x^2-1)H_{2} + x(x^2-2)H_{3} + x^2H_{4} 
                    & x^2H_{1} - xH_{2} + (x^2-1)H_{3} + xH_{4}
                    & H_{4}
                    & J
            \end{array}
      \right).}
    \end{equation*}
    Subtracting the identity, we obtain $\psi_n(\beta \gen{i}\gen{i+1}\gen{i})-I_{n-1}$. Then, we apply the following sequence of elementary operations
    \begin{align*}
         \psi_{n}(\beta \gen{i}\gen{i+1}\gen{i}) - I_{n-1}  \xrightarrow{\scalemath{0.7}{C_{i,i+1} (-x)}}
        \bullet & \xrightarrow{\scalemath{0.7}{C_{i+1,i-1} (-x^2)}}
        \bullet \xrightarrow{\scalemath{0.7}{C_{i+1,i}(x)}} 
        \bullet \xrightarrow{\scalemath{0.7}{C_{i+1,i+2} (-x)}} 
        \cdots \\  \cdots\to
        \bullet & \xrightarrow{\scalemath{0.7}{C_{i,i+1} (-x)}}
        \bullet \xrightarrow{\scalemath{0.7}{C_{i+1} (-1)}} M,
    \end{align*}
    where $M$ coincides with $\psi_n(\beta)-I_{n-1}$ except for the $i$-th and $(i+1)$-th columns. More precisely, if we write $\psi_n(\beta)-I_{n-1}=(a_1,a_2,\dots,a_{n-1})$ and $M=(m_1,m_2,\dots,m_{n-1})$ in vector columns $a_i, m_i\in \Lambda^{n-1}$ for $1\leq i \leq n-1$, then
    \begin{equation*}
        m_j=\left\{\begin{array}{cl}
            a_i + xu & \textrm{ for } j=i,\\
            a_{i+1} + u & \textrm{ for } j=i+1,\\
            a_j & \textrm{ otherwise},
        \end{array}\right. \qquad \textrm{where }
        u= \scalemath{0.8}{\left(\begin{array}{c}
             0 \\
             \hline
             -x^2 \\
             x \\
             -x^2+2 \\
             -x \\
             \hline
             0
        \end{array}\right)\begin{array}{l}
              \\
              \\
             \raisebox{3pt}{{\scriptsize\mbox{{$i$-th}}}}\\
             \\
             \\
             \\
        \end{array}.}
    \end{equation*}
    By the sequence of elementary operations applied to $\psi_n(\beta\gen{i}\gen{i+1}\gen{i})-I_{n-1}$, it follows that $\det(\psi_n(\beta\gen{i}\gen{i+1}\gen{i})-I_{n-1})=-\det(M).$
    
    On the other hand, for the term $\psi_n(\gen{i+1}\gen{i}\gen{i+1})$ given by
    \[
     \psi_n(\gen{i+1}\gen{i}\gen{i+1}) = \left(
            \begin{array}{c | c  c  c  c | c }
                I_{i-2} & \multicolumn{4}{c|}{0} & 0\\
                 \hline
                 & 1 & x & x^2 & 0 & \\
                 0 & 0 & x^2-1 & x(x^2-2) & 0 & 0 \\
                 & 0 & -x & -(x^2-1) & 0 & \\
                 & 0 & x^2 & x^3 & 1 & \\
                 \hline
                 0  & \multicolumn{4}{c|}{0} & I_{n-i-3}\\
            \end{array}
        \right),
    \]
    the resulting product of $\psi_n(\beta)$ and $\psi_n(\gen{i+1}\gen{i}\gen{i+1})$ is
    \[
    \scalemath{0.7}{
        \left(
            \begin{array}{c | c | c | c | c | c }
                A   & B_{1}  
                    & xB_{1} + (x^2-1)B_{2} - xB_{3} + x^2B_{4} 
                    & x^2B_{1} + x(x^2-2)B_{2} - (x^2-1)B_{3} + x^3B_{4}
                    & B_{4}
                    & C \\
                \hline
                & \vdots & \vdots & \vdots & \vdots & \\
                D   & E_{k,1}  
                    & xE_{k,1} + (x^2-1)E_{k,2} - xE_{k,3} + x^2E_{k,4} 
                    & x^2E_{k,1} + x(x^2-2)E_{k,2} - (x^2-1)E_{k,3} + x^3E_{k,4}
                    & E_{k,4}
                    & F \\
                & \vdots & \vdots & \vdots & \vdots & \\
                \hline
                G   & H_{1}  
                    & xH_{1} + (x^2-1)H_{2} - xH_{3} + x^2H_{4} 
                    & x^2H_{1} + x(x^2-2)H_{2} - (x^2-1)H_{3} + x^3H_{4}
                    & H_{4}
                    & J
            \end{array}
       \right).
       }
    \]
    In the same way as before, we subtract the identity and apply the following sequence of elementary operations
    \begin{align*}
         \psi_{n}(\beta \gen{i+1}\gen{i}\gen{i+1}) - I_{n-1}  \xrightarrow{\scalemath{0.7}{C_{i+1,i} (-x)}} 
        \bullet & \xrightarrow{\scalemath{0.7}{C_{i,i-1} (-x)}}
        \bullet \xrightarrow{\scalemath{0.7}{C_{i,i+1}(x)}}
        \bullet \xrightarrow{\scalemath{0.7}{C_{i,i+2} (-x^2)}} 
        \cdots \\  \cdots\to
        \bullet & \xrightarrow{\scalemath{0.7}{C_{i+1,i} (-x)}}
        \bullet \xrightarrow{\scalemath{0.7}{C_{i} (-1)}} M',
    \end{align*}
    where $M'$ coincides with $\psi_n(\beta)-I_{n-1}$ except for the $i$-th and $(i+1)$-th columns. Namely, $M'=(m_{1}',\dots,m_{n-1}')$ with
    \begin{equation*}
        m_j'=\left\{\begin{array}{cl}
            a_i + v & \textrm{ for } j=i,\\
            a_{i+1} + xv & \textrm{ for } j=i+1,\\
            a_j & \textrm{ otherwise},
        \end{array}\right. \qquad \textrm{where }
        v= \scalemath{0.8}{\left(\begin{array}{c}
             0 \\
             \hline
             -x \\
             -x^2+2 \\
             x \\
             -x^2 \\
             \hline
             0
        \end{array}\right)\begin{array}{l}
              \\
              \\
             \raisebox{3pt}{{\scriptsize\mbox{{$i$-th}}}}\\
             \\
             \\
             \\
        \end{array}.}
    \end{equation*}
    The determinant is computed as
    $\det(\psi_n(\beta\gen{i+1}\gen{i}\gen{i+1})-I_{n-1})=-\det(M').$ Consequently, on the left-hand side of Equation (\ref{eq:equiv skein}) we have
    \begin{equation}\label{eq:left skein equiv}
    \scalemath{0.9}{
        \det(\psi_n(\beta\gen{i}\gen{i+1}\gen{i})-I_{n-1})- \det(\psi_n(\beta\gen{i+1}\gen{i}\gen{i+1})-I_{n-1}) = -\det(M)+\det(M').
        }
    \end{equation}
    Recalling that $\psi_n(\beta)-I_{n-1}=(a_1,\dots,a_{n-1})$, in what follows we write dots for entries $a_j$ with $j\neq i,i+1$. Thus we get
    \begin{align}
        \scalemath{0.9}{-\det(M)+\det(M')} &= \scalemath{0.9}{-\det(\dots,a_{i}+xu,a_{i+1}+u,\dots) + \det(\dots,a_{i}+v,a_{i+1}+xv,\dots)} \nonumber\\
        &= \scalemath{0.9}{\det(\dots,v-xu,a_{i+1},\dots) + \det(\dots,a_{i},xv-u,\dots)} \nonumber\\
        &=\scalemath{0.9}{(x^2-1)[\det(\dots,w,a_{i+1},\dots) - \det(\dots,a_{i},w',\dots)],} \label{eq: M M'}
    \end{align}
    where $w,w'$ are column vectors given by
    \begin{equation}\label{eq:w and w'}
        w= \scalemath{0.8}{\left(\begin{array}{c}
             0 \\
             \hline
             x \\
             -2 \\
             x \\
             0 \\
             \hline
             0
        \end{array}\right)\begin{array}{l}
              \\ \\ \raisebox{3pt}{{\scriptsize\mbox{{$i$-th}}}} \\ \\ \\ \\
        \end{array}} 
        \qquad \textrm{ and } \qquad
        w'= \scalemath{0.8}{\left(\begin{array}{c}
             0 \\
             \hline
             0 \\
             x \\
             -2 \\
             x \\
             \hline
             0
        \end{array}\right)\begin{array}{l}
              \\ \\ \raisebox{3pt}{{\scriptsize\mbox{{$i$-th}}}} \\ \\ \\ \\
        \end{array}.}
    \end{equation}
    
    Now we work the right-hand side of Equation (\ref{eq:equiv skein}). The resulting product of $\psi_n(\beta)$ and $\psi_n(\gen{i})$ is:
    \begin{equation*}
    \scalemath{0.9}{
     \psi_{n}(\beta\gen{i})=
        \left(
            \begin{array}{c | c | c | c | c | c }
                A & B_{1}  
                    & xB_{1} - B_{2} + xB_{3}
                    & B_{3}
                    & B_{4}
                    & C \\
                \hline
                & \vdots & \vdots & \vdots & \vdots & \\
                D   & E_{k,1}  
                    & xE_{k,1} - E_{k,2} + xE_{k,3}
                    & E_{k,3}
                    & E_{k,4}
                    & F \\
                & \vdots & \vdots & \vdots & \vdots & \\
                \hline
                G   & H_{1}  
                    & xH_{1} - H_{2} + xH_{3}
                    & H_{3}
                    & H_{4}
                    & J
            \end{array}
       \right).
    }
    \end{equation*}
    In a similar fashion, we subtract the identity and apply the following sequence of elementary operations
    \begin{align*}
         \psi_{n}(\beta \gen{i}) - I_{n-1}  \xrightarrow{\scalemath{0.7}{C_{i,i-1} (-x)}} 
        \bullet & \xrightarrow{\scalemath{0.7}{C_{i,i+1} (-x)}}
        \bullet \xrightarrow{\scalemath{0.7}{C_{i}(-1)}} P,
    \end{align*}
    where $P$ coincides with $\psi_{n}(\beta) - I_{n-1}$ except for the $i$-th column. Namely, $P=(p_1,\dots,p_{n-1})$ with
    \begin{equation*}
        p_j=\left\{\begin{array}{cl}
            a_i - w & \textrm{ for } j=i,\\
            a_j & \textrm{ otherwise},
        \end{array}\right. \qquad (1\leq j \leq n-1)
    \end{equation*}
    where $w$ is as in Equation (\ref{eq:w and w'}). By the sequence of elementary operations, it follows that $\det(\psi_n(\beta\gen{i})-I_{n-1}) =-\det(P)$.
    
    For the other term in the right-hand side of (\ref{eq:equiv skein}), the resulting product of $\psi_n(\beta)$ and $\psi_n(\gen{i+1})$ is:
    \begin{equation*}
    \scalemath{0.9}{
     \psi_{n}(\beta\gen{i+1})=
        \left(
            \begin{array}{c | c | c | c | c | c }
                A & B_{1}  
                    & B_{2}
                    & xB_{2} - B_{3} + xB_{4}
                    & B_{4}
                    & C \\
                \hline
                & \vdots & \vdots & \vdots & \vdots & \\
                D   & E_{k,1}  
                    & E_{k,2}
                    & xE_{k,2} - E_{k,3} + xE_{k,4}
                    & E_{k,4}
                    & F \\
                & \vdots & \vdots & \vdots & \vdots & \\
                \hline
                G   & H_{1}  
                    & H_{2}
                    & xH_{2} - H_{3} + xH_{4}
                    & H_{4}
                    & J
            \end{array}
       \right).
    }
    \end{equation*}
    Following the similar procedure of subtract the identity and applying a sequence of elementary operations, we have
    \begin{align*}
         \psi_{n}(\beta \gen{i}) - I_{n-1}  \xrightarrow{\scalemath{0.7}{C_{i+1,i} (-x)}}
        \bullet & \xrightarrow{\scalemath{0.7}{C_{i+1,i+2} (-x)}}
        \bullet \xrightarrow{\scalemath{0.7}{C_{i+1}(-1)}} P',
    \end{align*}
    where $P'=(p'_1,\dots,p'_{n-1})$ with
    \begin{equation*}
        p'_j=\left\{\begin{array}{cl}
            a_{i+1} - w' & \textrm{ for } j=i+1,\\
            a_j & \textrm{ otherwise},
        \end{array}\right. \qquad (1\leq j \leq n-1)
    \end{equation*}
    and $w'$ is as in Equation (\ref{eq:w and w'}). By the sequence of elementary operations applied to $\psi_{n}(\beta \gen{i+1}) - I_{n-1}$, its determinant is computed as $\det(\psi_n(\beta\gen{i+1})-I_{n-1})=-\det(P')$. Calculating the right-hand side of Equation (\ref{eq:equiv skein}), we obtain
    \begin{equation}\label{eq:right skein equiv}
        \det(\psi_{n}(\beta\gen{i})-I_{n-1}) - \det(\psi_{n}(\beta\gen{i+1})-I_{n-1}) = -\det(P) + \det(P').
    \end{equation}
    Simplifying,
    \begin{align*}
        \scalemath{1.0}{-\det(P)+\det(P')} &= \scalemath{1.0}{-\det(\dots,a_{i}-w,a_{i+1},\dots) + \det(\dots,a_{i},a_{i+1}-w',\dots)} \\
        &=\scalemath{1.0}{\det(\dots,w,a_{i+1},\dots) - \det(\dots,a_{i},w',\dots)} \\
        &=\scalemath{1.0}{\frac{-\det(M)+\det(M')}{(x^2-1)}},
    \end{align*}
    where the last equality is given by (\ref{eq: M M'}). Therefore, by Equations (\ref{eq:left skein equiv}), (\ref{eq:right skein equiv}) we obtain Equation (\ref{eq:equiv skein}), which is equivalent to the skein relation (\ref{eq:skein rel}).
\end{proof}

\subsection{The invariant $\mathcal{Q}$}\label{subsec:invariant Q}

In subsection \ref{subsec:f_n} we investigated how functions $f_n:\B{n}\to\Zx$ behave under Gotin-Markov moves. It is well summarized in the following pictures:
\begin{enumerate}[label=\roman*),itemsep=6mm]
        \item
            $f_{n+1}\big(\, {\risSpdf{-9}{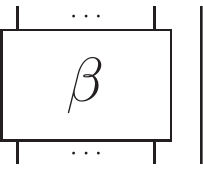}{}{25}{0}{0}} \, \big) = 0 = f_{n+1}\big( \, {\risSpdf{-9}{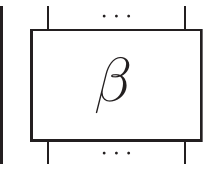}{}{25}{0}{0}}\, \big)$,
        \item
            $f_{n}\big(\, {\risSpdf{-22}{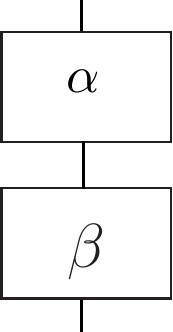}{}{25}{0}{0}} \, \big) = f_{n}\big( \, {\risSpdf{-22}{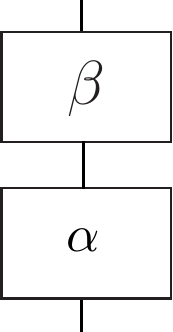}{}{25}{0}{0}} \, \big)$,
        \item
            $f_{n+1}\big( \, {\risSpdf{-12}{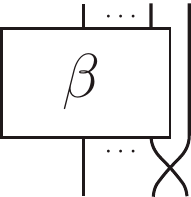}{}{25}{0}{0}} \, \big) = f_{n}\big( \,  {\risSpdf{-9}{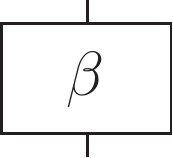}{}{25}{0}{0}} \, \big) =
            f_{n+1}\big( \, {\risSpdf{-12}{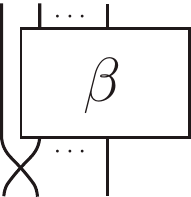}{}{25}{0}{0}} \, \big)$,
        \item
            $f_{n+1}\big( \, {\risSpdf{-22}{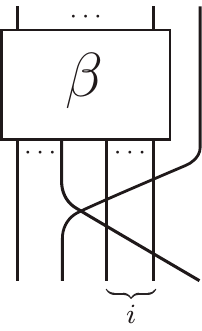}{}{25}{0}{0}} \, \big) = x^{2i} f_{n}\big( \,  {\risSpdf{-9}{beta}{}{25}{0}{0}} \, \big) =
            f_{n+1}\big( \, {\risSpdf{-22}{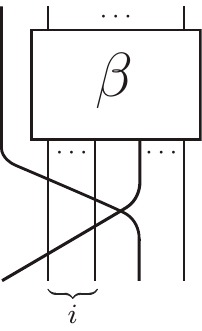}{}{25}{0}{0}} \, \big)$,
        \item
            $f_{n}\big( {\risSpdf{-9}{C121}{}{25}{0}{0}} \big) - f_{n}\big( {\risSpdf{-9}{C212}{}{25}{0}{0}} \big) = (x^2-1)\left(f_{n}\big( {\risSpdf{-9}{C1}{}{25}{0}{0}} \big) - f_{n}\big( {\risSpdf{-9}{C2}{}{25}{0}{0}} \big) \right)$,
    \end{enumerate}
    \vspace{2mm}
    for any $\alpha,\beta \in \B n$.
    
    From properties above, it is clear that if we have two equivalent doodles $D \sim D'$, and $\alpha,\beta$ two twins with closures $D$ and $D'$ respectively. Then, $\alpha \sim_G \beta$ are Gotin-Markov equivalent and the corresponding polynomials $f_n(\alpha)$ and $f_m(\beta)$ are the same up to a factor $x^{2k}$ for some integer $k$. Furthermore, by the skein relation (\ref{eq:skein rel}), all the polynomials of a doodle $D$ are polynomials in $\mathbb{Z}[x^2]$. In this way, it is natural to define our invariant as follows.
    
\begin{definition}
For a doodle $D\in \mathcal{D}$, let $I_D$ be the ideal in $\mathbb{Z}[x^2]$ generated by all the polynomials $f_n(\alpha)$ with $\hat{\alpha}\sim D$. We define $\mathcal{Q}:\mathcal{D}\longrightarrow \mathbb{Z}[x^2]$ given by
\begin{equation}
    \mathcal{Q}(D):=\gcd(I_D),
\end{equation}
where $\gcd$ is the generator of the ideal $I_D$.
\end{definition}

As a direct consequence of all the previous discussion, we have our main result.

\begin{theorem}\label{thm:inv Q}
    The function $\mathcal{Q}$ is a polynomial invariant of oriented doodles. 
\end{theorem}

\begin{remark}
    Despite the function $f_n$ satisfies the skein relation (\ref{eq:skein rel}), it is important to notice that the invariant $\mathcal{Q}$ does not. It is a consequence of reducing the degrees of the polynomials by taking the $\gcd$.
\end{remark} 

Let us calculate a simple example of the invariant $\mathcal{Q}$.

\begin{example}
    According to \cite[Theorem~4.2]{BFKK18}, the first non-trivial doodle has 6 crossings, and is the Borromean doodle (see Figure \ref{fig:borromean}). Let us compute its polynomial under $\mathcal{Q}$. In Example \ref{example:f(t1t2)^3}, we compute two associated polynomials to the Borromean doodle through two Gotin-Markov equivalent twins. Furthermore, for any twin $\beta$ with closure the Borromean doodle, we have that $\deg (f_3((\gen{1}\gen{2})^3)) \leq \deg (f_m(\beta))$.
    Thus,
    \vspace{3mm}
    
    \begin{equation*}
        \mathcal{Q}\big( \, \risS{-.8\height}{borromean}{}{35}{0}{0} \, \big) = x^4 - 2x^2 + 1.
    \end{equation*}
    
    \vspace{2mm}
\end{example}

\begin{remark}
    For a doodle $D$, if there exists a twin $\alpha_0$ with closure the minimal representative $D_0$ of $D$, then $\mathcal{Q}(D)$ is obtained with the corresponding twin $\alpha_0$, i.e., $\mathcal{Q}(D)= f_n(\alpha_0)$. In general, not every minimal representative of a doodle is realizable immediately by the closure of a twin, might be necessary to produce new double points in the diagram.
\end{remark}

Analogously to the classical Alexander polynomial of links, the invariant $\mathcal{Q}$ vanishes on multi-component doodles with `unlinked' components.

\begin{proposition}
    Let $D=D_1 \bigsqcup D_2$ be a doodle with two disjoint components, then $\mathcal{Q}(D)=0$.
\end{proposition}

\begin{proof}
    If one of the components is the trivial doodle, the result follows easily from the skein relation (\ref{eq:skein rel}), or by definition of $f_n$ and Lemma \ref{lem:fn M0 y M1}. Let $\alpha\in \B{n}$ and $\beta\in\B{m}$ be two twins such that $\widehat{\alpha^{-1}}=D_1$ and $\widehat{\beta}=D_2$, then adding vertical strands on the left or right, we have $\iota^R(\alpha),\iota^L(\beta) \in \B{n+m}$ such that $\widehat{\iota^R(\alpha)\iota^L(\beta)} = \widehat{\iota^L(\beta)\iota^R(\alpha)} = D$.
    By Equation (\ref{eq:matrix *beta}) and Remark \ref{rem:b[1]} it follows
    \begin{equation*}
        \psi_{n+m}(\iota^R(\alpha)\iota^L(\beta))=\left(\begin{array}{ccc}
            \psi_n(\alpha) & 0 & 0 \\
             v^R(\alpha) & 1 & v^L(\beta) \\
             0 & 0 & \psi_m(\beta)
        \end{array}\right).
    \end{equation*}
    Thus, the $n$-th column of $\psi_{n+m}(\iota^R(\alpha)\iota^L(\beta))-I_{n+m-1}$ is zero and $f_{n+m}(\iota^R(\alpha)\iota^L(\beta))$ vanishes. Since every polynomial associated to the doodle is the same, up to a factor $x^{2k}$ for some integer $k$, it follows $\mathcal{Q}(D)=0$.
\end{proof}

\section{Some Computations}\label{sec:Computations}

In this section we compute the invariant $\mathcal{Q}$ for many doodles found in the literature \cite{BFKK18,FT79,Kh97,Mer99-2,Mer03,V99}. The results are summarized in a table in Subsection \ref{subsec:table}.

\subsection{Families of doodles}

In \cite{BFKK18}, Bartholomew-Fenn-Kamada-Kamada give a geometric construction of infinite families of doodles from polygons. We only recall one of those examples in detail and the rest is described vaguely or in terms of twins. We refer to \cite{BFKK18} for the original constructions.

\begin{example}
    Define the doodle $B_n$ as follows. Start with two concentric $n$-gons with vertices $X_1,\dots,X_n$ and $Y_1,\dots,Y_n$, respectively, and add the missing edges to construct the triangles $X_{i}Y_{i}X_{i+1}$ for $1\leq i\leq n$ cyclically modulo $n$ (see Figure \ref{fig:n-poppy}). All vertices of the resulting diagram have valency four, and it is natural how to smooth the edges to obtain the doodle $B_n$. By the symmetry of the diagram, it is easy to see that $\widehat{(\gen{1}\gen{2})^n}$ is $B_n$ for any $n$. In particular, when $n$ is divisible by 3, the twin $(\gen{1}\gen{2})^n$ induces the trivial permutation, i.e., is an element of the pure twin group in 3 strands. Thus, the number of components of $B_n$ is three if $n$ is divisible by 3, called the \textit{$n$-generalized Borromean} doodle; otherwise, $B_n$ has one component, called the \textit{$n$-poppy} doodle.
    
    Computations prompt the invariant $\mathcal{Q}$ never vanishes on this family and its restriction to the family is very strong such that distinguishes, i.e., if $\mathcal{Q}(B_n)=\mathcal{Q}(B_m)$, then $n=m$.
\end{example}

\begin{figure}[ht]
    \begin{minipage}[t]{.45\textwidth}
        \centering
         \includegraphics[scale=1]{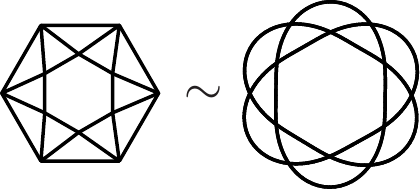}
        \subcaption{6-Borromean}
    \end{minipage}
    \hfill
    \begin{minipage}[t]{.45\textwidth}
        \centering
        \includegraphics[scale=1]{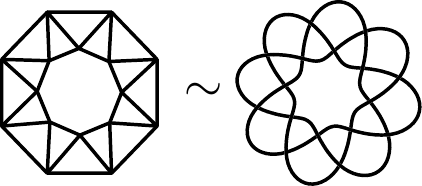}
        \subcaption{8-poppy}
    \end{minipage}
    \caption{Generalized Borromean and $n$-poppy doodles.} \label{fig:n-poppy}
\end{figure}

\begin{example}
    The doodle $C^1_n$ is constructed from the previous example $D_n$ adding a circle as in Figure \ref{subfig:C^1_3}. In terms of twins, $C^1_n$ is the closure of the twin $(\gen{1}\gen{2}\gen{3}\gen{2})^n$. The doodle $C^1_n$ has four components if $n$ is divisible by 3, otherwise $C^1_n$ has two components.
    
    The doodle $C^1_n$ can be generalized to the doodle $D_n$ with $r$ circles added. Written in twins, define the doodle $C^r_n$ as the closure of the twin $(\gen 1 (\gen 2 \cdots \gen {r+1}) \gen {r+2} (\gen {r+1} \cdots \gen 2) )^n$. In this case, the doodle $C^r_n$ has $r+3$ components if $n$ is divisible by 3, otherwise $C^r_n$ has $r+1$ components.
\end{example}

\begin{figure}[ht]
    \begin{minipage}[t]{.3\textwidth}
        \centering
         \includegraphics[scale=1]{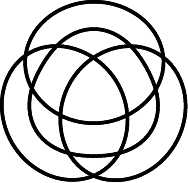}
        \subcaption{$C^1_3$}\label{subfig:C^1_3}
    \end{minipage}
    \hfill
    \begin{minipage}[t]{.3\textwidth}
        \centering
        \includegraphics[scale=1]{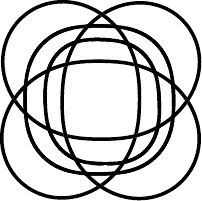}
        \subcaption{$C^2_4$}
    \end{minipage}
    \hfill
    \begin{minipage}[t]{.3\textwidth}
        \centering
         \includegraphics[scale=1]{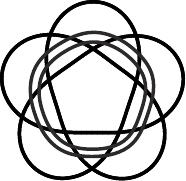}
        \subcaption{$C^2_5$}
    \end{minipage}
    \caption{Doodles $C^r_n$.}\label{fig:Cn with circles}
\end{figure}

Testing the invariant under different families of doodles, unfortunately we found $\mathcal{Q}$ vanishes for an infinite family of doodles.

\begin{proposition}\label{prop:Cnr = 0}
    For integers $r\geq 1$ and $n\geq 3$, $\mathcal{Q}(C^r_n)=0$.
\end{proposition}
\begin{proof}
    The doodle $C^r_n$ is the closure of the twin $\alpha_r^n=(\gen 1 (\gen 2 \cdots \gen {r+1}) \gen {r+2} (\gen {r+1} \cdots \gen 2) )^n$. Applying induction over $r$ and using iteratively Lemma \ref{lem:braid rel estimation}, the matrix $\psi_{r+3}(\alpha_r)^n - I_{r+2}$ has $r+1$ columns equal up to a power of $x$. More precisely,
    $$\psi_{r+3}(\alpha_r)^n - I_{r+2}=(c_1, x^{r}c_{r+2},x^{r-1}c_{r+2},\dots,xc_{r+2},c_{r+2}),$$ 
    written in columns. Hence $\det(\psi_{r+3}(\alpha_r)^n - I_{r+2})=0$ and the result follows.
\end{proof}

\begin{example}
    Another infinite family of doodles related to $C^r_n$ (by the geometric construction given in \cite{BFKK18}) is the family $D^1_n$. It consists of a circle as the skeleton and an array of $n$ trivial doodles $O_1,\dots,O_n$ overlapping as in Figures \ref{subfig:D_3^1},\ref{subfig:D_4^1}. Doodles $D^r_n$ are a natural generalization of $D^1_n$ by adding $r$ concentric circles in the skeleton. We do not know a general formula of $D^r_n$ written in twins, but for small values of $r$ and $n$, we have the invariant $\mathcal{Q}$ vanishes.

\begin{figure}[ht]
    \begin{minipage}[t]{.3\textwidth}
        \centering
         \includegraphics[scale=1]{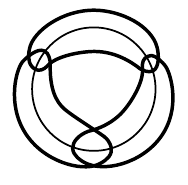}
        \subcaption{$D^1_3$} \label{subfig:D_3^1}
    \end{minipage}
    \hfill
    \begin{minipage}[t]{.3\textwidth}
        \centering
        \includegraphics[scale=1]{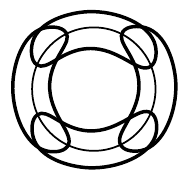}
        \subcaption{$D^1_4$} \label{subfig:D_4^1}
    \end{minipage}
    \hfill
    \begin{minipage}[t]{.3\textwidth}
        \centering
         \includegraphics[scale=1]{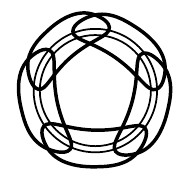}
        \subcaption{$D^2_5$}
    \end{minipage}
    \caption{Doodles $D^r_n$.} \label{fig:Dn with circles}
\end{figure}
    
\end{example}

\subsection{Table}\label{subsec:table}
We follow a format very similar to the table of knots given in \cite{Ada04}. The number in the first line following the picture of the doodle, stands for the number of crossings and components of the minimal representative. The second line is the twin representation of the doodle, where the number $k$ denotes the generator $\gen{k}$. The last line codified the polynomial invariant $\mathcal{Q}(D) \in \mathbb{Z}[x^2]$, the first number in curly brackets denotes the half of the maximum degree of the polynomial and the next sequence in parenthesis denotes the coefficients in even powers, from higher to lower degree. For instance, $\{7\}(1,2,-1,-2,1)$ denotes the polynomial $x^{14} + 2x^{12} - x^{10} - 2x^{8} + x^{6}$. 

\begin{center}
    \textbf{One component doodles}
\end{center}

\begin{multicols}{2}

\noindent
$\scalemath{1}{{\tabulinesep=2mm
\begin{tabu}{c l}
    \raisebox{-.5\height}{\includegraphics[width=2.5cm]{trivial}} &  \begin{array}{l} : 0^1 \\ : 1 \in\B 1 
    \end{array}\\
    \multicolumn{2}{l}{ : \{0\}(1) } \\ 
\end{tabu}}}$

\noindent
$\scalemath{1}{{\tabulinesep=2mm
\begin{tabu}{c l}
    \raisebox{-.5\height}{\includegraphics[width=2.5cm]{4poopy}} &  \begin{array}{l} : 8^1 \\ : (1 2)^4  \end{array}\\
    \multicolumn{2}{l}{ : \{3\}(1,-4,4) } \\ 
\end{tabu}}}$

\noindent
$\scalemath{1}{{\tabulinesep=2mm
\begin{tabu}{c l}
    \raisebox{-.5\height}{\includegraphics[width=2.5cm]{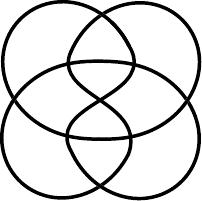}} &  \begin{array}{l} : 9^1 \\ : ( 3  2  1)^3  \end{array}\\
    \multicolumn{2}{l}{ : \{2\}(4,-4,1) } \\ 
\end{tabu}}}$

\noindent
$\scalemath{1}{{\tabulinesep=2mm
\begin{tabu}{c l}
    \raisebox{-.5\height}{\includegraphics[width=2.5cm]{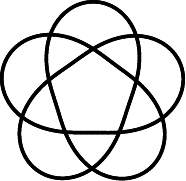}} &  \begin{array}{l} : 10^1 \\ : ( 1  2)^5 \end{array}\\
    \multicolumn{2}{l}{: \{4\}(1,-6,11,-6,1)} \\ 
\end{tabu}}}$

\noindent
$\scalemath{1}{{\tabulinesep=2mm
\begin{tabu}{c l}
    \raisebox{-.5\height}{\includegraphics[width=2.5cm]{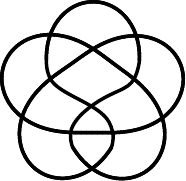}} &  \begin{array}{l} : 11^1 \\ : (21)^2(23)^2 123 \end{array}\\
    \multicolumn{2}{l}{ : \{4\}(1,-2,3,-2,1) } \\ 
\end{tabu}}}$

\noindent
$\scalemath{1}{{\tabulinesep=2mm
\begin{tabu}{c l}
    \raisebox{-.5\height}{\includegraphics[width=2.5cm]{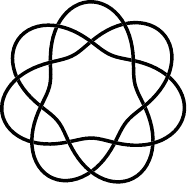}} &  \begin{array}{l} :  14^1 \\ : ( 1  2)^7  \end{array}\\
    \multicolumn{2}{l}{ : \{6\}(1,-10,37,-62,46,-12,1) } \\ 
\end{tabu}}}$

\noindent 
$\scalemath{1}{{\tabulinesep=2mm
\begin{tabu}{c l}
    \raisebox{-.5\height}{\includegraphics[width=2.5cm]{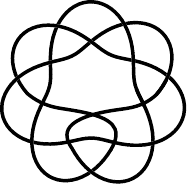}} &  \begin{array}{l} :  15^1 \\ :  54321 (432)^2 54321 \\ 2 (43)^2 \end{array}\\
    \multicolumn{2}{l}{ : \{7\}(1,2,-1,-2,1) } \\ 
\end{tabu}}}$

\noindent
$\scalemath{1}{{\tabulinesep=2mm
\begin{tabu}{c l}
    \raisebox{-.5\height}{\includegraphics[width=2.5cm]{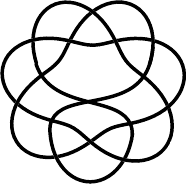}} &  \begin{array}{l} :  15^1 \\ : (12)^3 321 (32)^2  \end{array}\\
    \multicolumn{2}{l}{ : \{6\}(1,-6,13,-14,10,-4,1) } \\ 
\end{tabu}}}$

\noindent
$\scalemath{1}{{\tabulinesep=2mm
\begin{tabu}{c l}
    \raisebox{-.5\height}{\includegraphics[width=2.5cm]{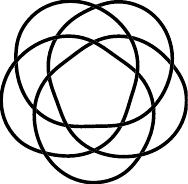}} &  \begin{array}{l} : 15^1 \\ : ( 3  2 1)^5  \end{array}\\
    \multicolumn{2}{l}{ : \{4\}(16,-48,44,-12,1) } \\ 
\end{tabu}}}$

\noindent
$\scalemath{1}{{\tabulinesep=2mm
\begin{tabu}{c l}
    \raisebox{-.5\height}{\includegraphics[width=2.5cm]{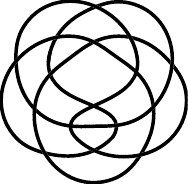}} &  \begin{array}{l} :  16^1 \\ : 321 (3214)^2 321 34  \end{array}\\
    \multicolumn{2}{l}{ : \{5\}(1,4,0,-8,4) } \\ 
\end{tabu}}}$

\noindent
$\scalemath{1}{{\tabulinesep=2mm
\begin{tabu}{c l}
    \raisebox{-.5\height}{\includegraphics[width=2.5cm]{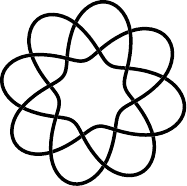}} &  \begin{array}{l} : 16^1 \\ : ( 1  2)^8 \end{array}\\
    \multicolumn{2}{l}{ : \{7\}(1,-12,56,-128,148,-80,16) } \\ 
\end{tabu}}}$

\noindent
$\scalemath{1}{{\tabulinesep=2mm
\begin{tabu}{c l}
    \raisebox{-.5\height}{\includegraphics[width=2.5cm]{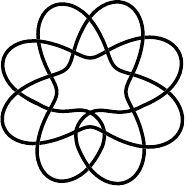}} &  \begin{array}{l} :  17^1 \\ : (4 3)^2 5 4 3 2 1 (4 3)^2  \\  (2 3)^2 5 4 3 2 1 2  \end{array}\\
    \multicolumn{2}{l}{ : \{9\}(1,- 2,1,2,- 2,1) } \\ 
\end{tabu}}}$

\noindent
$\scalemath{1}{{\tabulinesep=2mm
\begin{tabu}{c l}
    \raisebox{-.5\height}{\includegraphics[width=2.5cm]{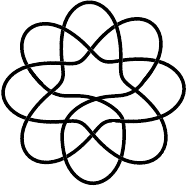}} &  \begin{array}{l} :  17^1 \\ : (12)^2 321 (32)^2 (12)^3 \end{array}\\
    \multicolumn{2}{l}{ : \{7\}(1,- 8,24,- 36,32,- 16,4) } \\ 
\end{tabu}}}$

\noindent
$\scalemath{1}{{\tabulinesep=2mm
\begin{tabu}{c l}
    \raisebox{-.5\height}{\includegraphics[width=2.5cm]{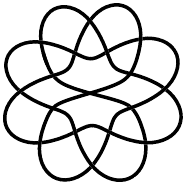}} &  \begin{array}{l} :  17^1 \\ : 7654321(43)^2 23 54 \\ 654 7654321 432 45 \\ 434 65 45 \end{array}\\
    \multicolumn{2}{l}{ : \{14\}(1,- 4,4) } \\ 
\end{tabu}}}$

\noindent
$\scalemath{1}{{\tabulinesep=2mm
\begin{tabu}{c l}
    \raisebox{-.5\height}{\includegraphics[width=2.5cm]{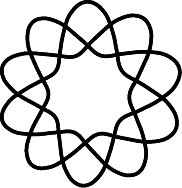}} &  \begin{array}{l} : 20^1 \\ : ( 1  2)^{10}  \end{array}\\
    \multicolumn{2}{l}{ : \{9\}(1,-16,106,-376,771,-920, } \\ 
    \multicolumn{2}{l}{ 610,-200,25) }
\end{tabu}}}$

\noindent
$\scalemath{1}{{\tabulinesep=2mm
\begin{tabu}{c l}
    \raisebox{-.5\height}{\includegraphics[width=2.5cm]{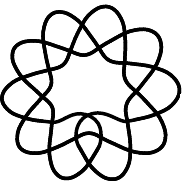}} &  \begin{array}{l} : 21^1 \\ : 12343 56 543212 \\ (543)^2 765432 (43)^2 \\ 5676 5434 \end{array}\\
    \multicolumn{2}{l}{ : \{14\}(1,2,- 2,1,- 2,1) } \\ 
\end{tabu}}}$

\noindent
$\scalemath{1}{{\tabulinesep=2mm
\begin{tabu}{c l}
    \raisebox{-.5\height}{\includegraphics[width=2.5cm]{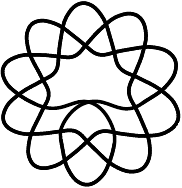}} &  \begin{array}{l} : 21^1 \\ : (12)^3 321323 (21)^4 2  \end{array}\\
    \multicolumn{2}{l}{ : \{9\}(1,- 12,58,- 148,223,- 212, } \\ %
    \multicolumn{2}{l}{ 130,- 48,9) } \\ 
\end{tabu}}}$

\noindent
$\scalemath{1}{{\tabulinesep=2mm
\begin{tabu}{c l}
    \raisebox{-.5\height}{\includegraphics[width=2.5cm]{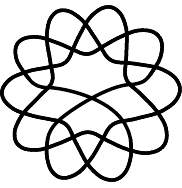}} &  \begin{array}{l} : 21^1 \\ : 543212 (43)^2 54321 \\ (43)^2 (23)^4 \end{array}\\
    \multicolumn{2}{l}{ : \{11\}(1,- 6,11,- 4,- 7,8,- 1,- 2,1) } \\ 
\end{tabu}}}$

\noindent
$\scalemath{1}{{\tabulinesep=2mm
\begin{tabu}{c l}
    \raisebox{-.5\height}{\includegraphics[width=2.5cm]{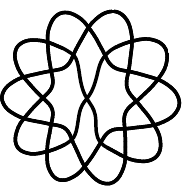}} &  \begin{array}{l} : 21^1 \\ : 7654323 (54)^2 654321 \\ 765432 (43)^2 5654321 \\ 543 \end{array} \\
    \multicolumn{2}{l}{ : \{12\}(4,-4,1) } \\ 
\end{tabu}}}$

\noindent
$\scalemath{1}{{\tabulinesep=2mm
\begin{tabu}{c l}
    \raisebox{-.5\height}{\includegraphics[width=2.5cm]{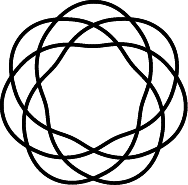}} &  \begin{array}{l} : 21^1 \\ : (1 2 3)^{7}  \end{array}\\
    \multicolumn{2}{l}{ :  \{6\}(64,- 320,592,- 496,184,- 24,1) } \\ 
\end{tabu}}}$

\noindent
$\scalemath{1}{{\tabulinesep=2mm
\begin{tabu}{c l}
    \raisebox{-.5\height}{\includegraphics[width=2.5cm]{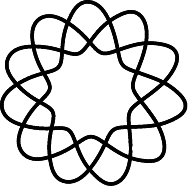}} &  \begin{array}{l} : 22^1 \\ : ( 1  2)^{11} \end{array}\\
    \multicolumn{2}{l}{ : \{10\}(1,-18,137,-574,1444, } \\ 
    \multicolumn{2}{l}{ -2232,2083,-1106,295,-30,1) } \\ 
\end{tabu}}}$

\noindent
$\scalemath{1}{{\tabulinesep=2mm
\begin{tabu}{c l}
    \raisebox{-.5\height}{\includegraphics[width=2.5cm]{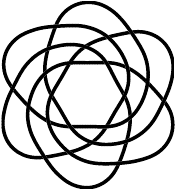}} &  \begin{array}{l} : 24^1 \\  : ( 1  2 3 4)^{6} \end{array}\\
    \multicolumn{2}{l}{ : \{10\}(1,-24,218,- 960,2251, } \\ 
    \multicolumn{2}{l}{ - 2880,1962,- 648,81) } \\ 
\end{tabu}}}$

\noindent
$\scalemath{1}{{\tabulinesep=2mm
\begin{tabu}{c l}
    \raisebox{-.5\height}{\includegraphics[width=2.5cm]{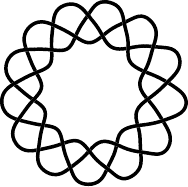}} &  \begin{array}{l} : 26^1 \\ : ( 1  2)^{13}  \end{array}\\
    \multicolumn{2}{l}{ : \{12\}(1,-22,211,-1158,4013, } \\ 
    \multicolumn{2}{l}{ -9142,13820,-13672,8518,-3108, } \\ 
    \multicolumn{2}{l}{ 581,-42,1) } \\ 
    \end{tabu}}}$

\noindent
$\scalemath{1}{{\tabulinesep=2mm
\begin{tabu}{c l}
    \raisebox{-.5\height}{\includegraphics[width=2.5cm]{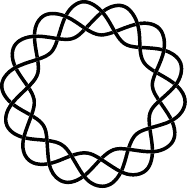}} &  \begin{array}{l} : 28^1 \\ : ( 1  2)^{14} \end{array}\\
    \multicolumn{2}{l}{: \{13\}(1,-24,254,-1560,6157, } \\ 
    \multicolumn{2}{l}{ - 16336, 29618,- 36568,30086, } \\ 
    \multicolumn{2}{l}{ - 15792,4900,- 784,49) } \\ 
\end{tabu}}}$

\noindent
$\scalemath{1}{{\tabulinesep=2mm
\begin{tabu}{c l}
    \raisebox{-.5\height}{\includegraphics[width=2.5cm]{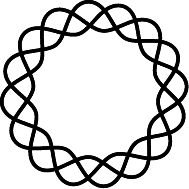}} &  \begin{array}{l} : 32^1 \\ : ( 1  2)^{16} \end{array}\\
    \multicolumn{2}{l}{: \{15\}(1,- 28,352,- 2624,12904, } \\ 
    \multicolumn{2}{l}{ - 44064,107104,- 186880,233108, } \\ 
    \multicolumn{2}{l}{ - 204528,122464,- 47616,11088, } \\ 
    \multicolumn{2}{l}{ - 1344,64) } \\ 
\end{tabu}}}$
\end{multicols}

\begin{center}
    \textbf{Multi-component doodles}    
\end{center}

\begin{multicols}{2}

\noindent
$\scalemath{1}{{\tabulinesep=2mm
\begin{tabu}{c l}
    \raisebox{-.5\height}{\includegraphics[width=2.5cm]{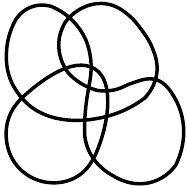}} &  \begin{array}{l} : 10^2 \\ : (12)^2 3 2 1 3 2 3 \end{array}\\
    \multicolumn{2}{l}{ : 0 } \\ 
\end{tabu}}}$

\noindent
$\scalemath{1}{{\tabulinesep=2mm
\begin{tabu}{c l}
    \raisebox{-.5\height}{\includegraphics[width=2.5cm]{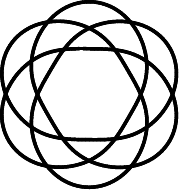}} &  \begin{array}{l} : 18^2 \\ : (1 2 3)^6 \end{array}\\
    \multicolumn{2}{l}{ : 0 } \\ 
\end{tabu}}}$

\noindent
$\scalemath{1}{{\tabulinesep=2mm
\begin{tabu}{c l}
    \raisebox{-.5\height}{\includegraphics[width=2.5cm]{borromean}} &  \begin{array}{l} : 6^3 \\  : ( 1  2)^{3} \end{array}\\
    \multicolumn{2}{l}{  : \{2\}(1,-2,1)  } \\ 
\end{tabu}}}$

\noindent
$\scalemath{1}{{\tabulinesep=2mm
\begin{tabu}{c l}
    \raisebox{-.5\height}{\includegraphics[width=2.5cm]{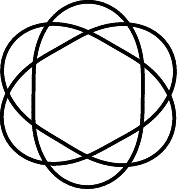}} &  \begin{array}{l} : 12^3 \\ : (1 2)^{6} \end{array}\\
    \multicolumn{2}{l}{ : \{5\}(1,-8,22,-24,9) } \\ 
\end{tabu}}}$

\noindent
$\scalemath{1}{{\tabulinesep=2mm
\begin{tabu}{c l}
    \raisebox{-.5\height}{\includegraphics[width=2.5cm]{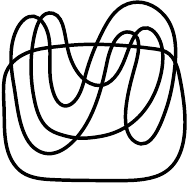}} &  \begin{array}{l} : 16^3 \\  : 4 3 2 3 1 2 3 2 4 3 2 1 \\ (2 3)^4  \end{array}\\
    \multicolumn{2}{l}{ : \{8\}(1,-6,15,-20,15,-6,1) } \\ 
\end{tabu}}}$ 

\noindent
$\scalemath{1}{{\tabulinesep=2mm
\begin{tabu}{c l}
    \raisebox{-.5\height}{\includegraphics[width=2.5cm]{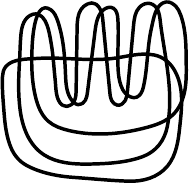}} &  \begin{array}{l} : 18^3 \\ :  ( 1  2)^{9} \end{array}\\
    \multicolumn{2}{l}{ : \{8\}(1,-14,79,-230,367,-314, } \\ 
    \multicolumn{2}{l}{ 130,-20,1) } \\ 
\end{tabu}}}$

\noindent
$\scalemath{1}{{\tabulinesep=2mm
\begin{tabu}{c l}
    \raisebox{-.5\height}{\includegraphics[width=2.5cm]{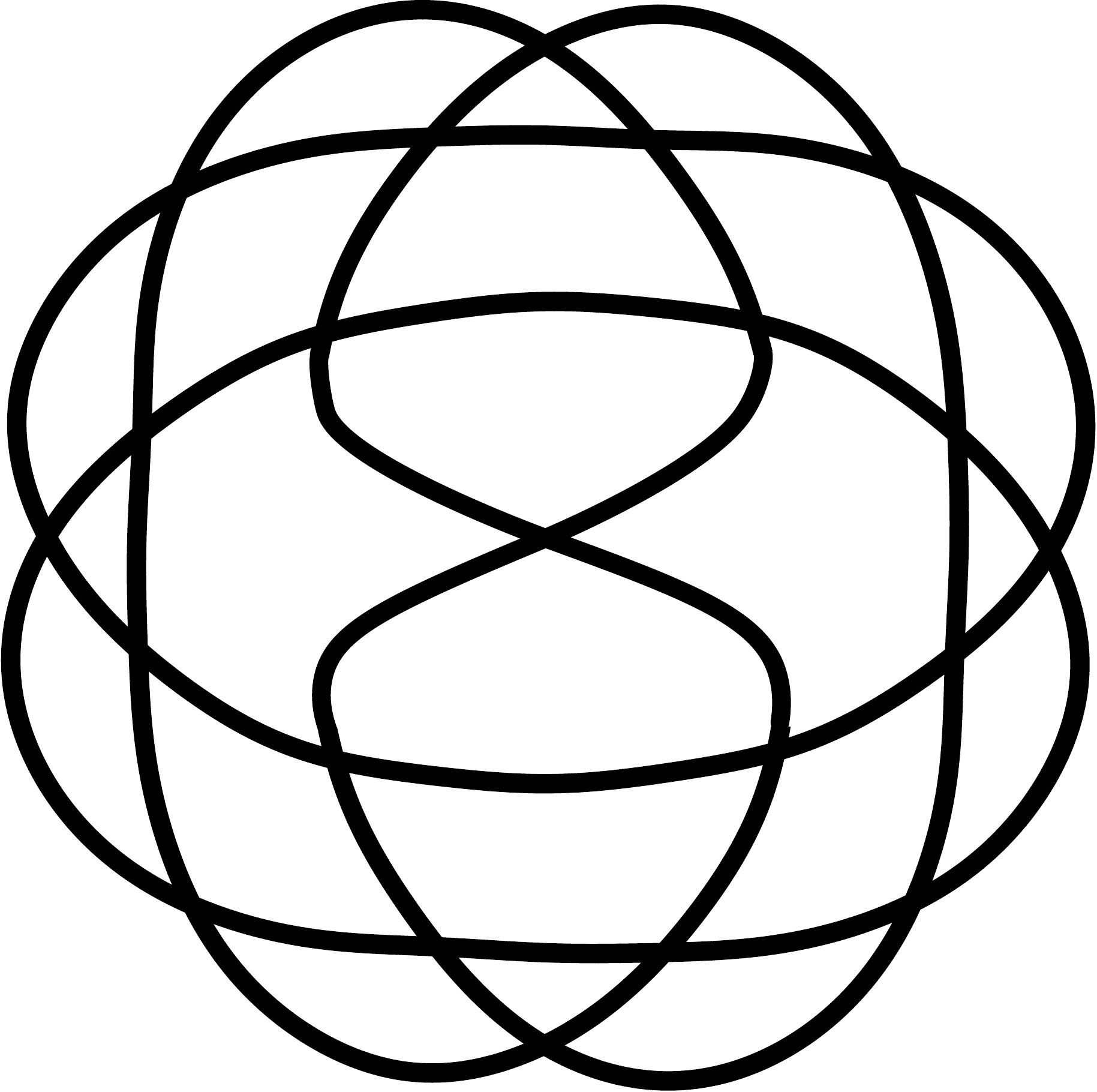}} &  \begin{array}{l} : 21^3 \\ : 12343 (23)^2 5 (432)^2 \\ 1234 3 543 234 \end{array}\\
    \multicolumn{2}{l}{ : \{11\}(1,- 4,4,4,- 10,4,4,- 4,1) } \\ 
\end{tabu}}}$

\noindent
$\scalemath{1}{{\tabulinesep=2mm
\begin{tabu}{c l}
    \raisebox{-.5\height}{\includegraphics[width=2.5cm]{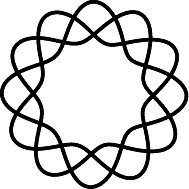}} &  \begin{array}{l} : 24^3 \\  : ( 1  2)^{12} \end{array}\\
    \multicolumn{2}{l}{ : \{11\}(1,-20,172,-832,2486,} \\ 
    \multicolumn{2}{l}{ - 4744,5776,-4532,1897,-420,36)} \\ 
\end{tabu}}}$

\noindent
$\scalemath{1}{{\tabulinesep=2mm
\begin{tabu}{c l}
    \raisebox{-.5\height}{\includegraphics[width=2.5cm]{D1_3}} &  \begin{array}{l} : 12^4 \\ : 54321 343 2 343 \\ 54321 323 43 \end{array}\\
    \multicolumn{2}{l}{ : 0 } \\ 
\end{tabu}}}$

\noindent
$\scalemath{1}{{\tabulinesep=2mm
\begin{tabu}{c l}
    \raisebox{-.5\height}{\includegraphics[width=2.5cm]{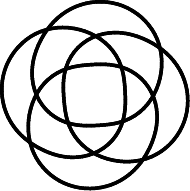}} &  \begin{array}{l} : 12^4 \\ : ( 1 2 3 )^{4} \end{array}\\
    \multicolumn{2}{l}{ : 0 } \\ 
\end{tabu}}}$

\noindent
$\scalemath{1}{{\tabulinesep=2mm
\begin{tabu}{c l}
    \raisebox{-.5\height}{\includegraphics[width=2.5cm]{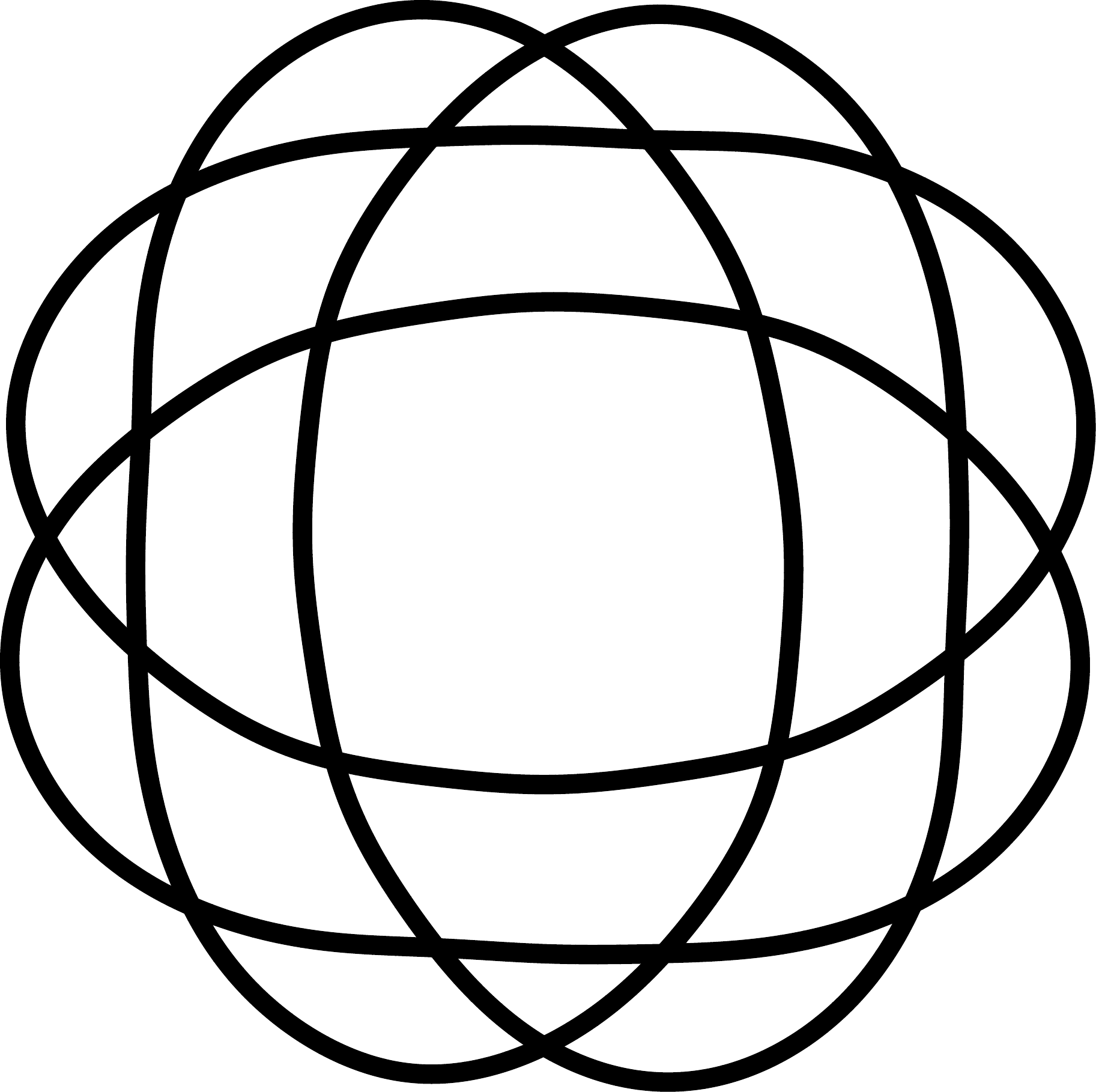}} &  \begin{array}{l} : 20^4 \\ : ( 1 3 2 3 2)^{4} \end{array}\\
    \multicolumn{2}{l}{ : 0 } \\ 
\end{tabu}}}$

\noindent
$\scalemath{1}{{\tabulinesep=2mm
\begin{tabu}{c l}
    \raisebox{-.5\height}{\includegraphics[width=2.5cm]{D1_4}} &  \begin{array}{l} : 16^5 \\  : 1234321 32123 432\\ 1 23 \end{array}\\
    \multicolumn{2}{l}{ : 0 } \\ 
\end{tabu}}}$

\noindent
$\scalemath{1}{{\tabulinesep=2mm
\begin{tabu}{c l}
     \raisebox{-.5\height}[2mm][0mm]{\Huge $C^{r-1}_n$} &  \begin{array}{l} : (2nr)^{r+2} \\ : ( 1 ( 2 \cdots  r)  {r+1} \\ ( r \cdots  2) )^n \\  \end{array}\\  
    \multicolumn{1}{l}{ : 0 }  & \textrm{ by Proposition \ref{prop:Cnr = 0}} \\ 
\end{tabu}}}$
\end{multicols}

\bibliographystyle{abbrv}
\bibliography{Biblio}

\begin{thebibliography}{10}

\bibitem{Ada04}
C.~C. Adams.
\newblock {\em The knot book}.
\newblock American Mathematical Society, Providence, RI, 2004.
\newblock An elementary introduction to the mathematical theory of knots,
  Revised reprint of the 1994 original.

\bibitem{BSV18}
V.~Bardakov, M.~Singh, and A.~Vesnin.
\newblock Structural aspects of twin and pure twin groups.
\newblock {\em Geom. Dedicata}, 203:135--154, 2019.

\bibitem{BarWeb}
A.~Bartholomew.
\newblock Virtual {D}oodles.
\newblock \url{http://www.layer8.co.uk/maths/doodles/index.htm}, 2018.
\newblock Updated: 2018-10-06.

\bibitem{BFKK18}
A.~Bartholomew, R.~Fenn, N.~Kamada, and S.~Kamada.
\newblock Doodles on surfaces.
\newblock {\em J. Knot Theory Ramifications}, 27(12):1850071, 26, 2018.

\bibitem{BFKK18-gauss}
A.~Bartholomew, R.~Fenn, N.~Kamada, and S.~Kamada.
\newblock On {G}auss codes of virtual doodles.
\newblock {\em J. Knot Theory Ramifications}, 27(11):1843013, 26, 2018.

\bibitem{Bour4-6}
N.~Bourbaki.
\newblock {\em Lie groups and {L}ie algebras. {C}hapters 4--6}.
\newblock Elements of Mathematics (Berlin). Springer-Verlag, Berlin, 2002.
\newblock Translated from the 1968 French original by Andrew Pressley.

\bibitem{Bur35}
W.~Burau.
\newblock \"{U}ber {Z}opfgruppen und gleichsinnig verdrillte {V}erkettungen.
\newblock {\em Abh. Math. Sem. Univ. Hamburg}, 11(1):179--186, 1935.

\bibitem{DG18}
S.~Dey and K.~Gongopadhyay.
\newblock Commutator subgroups of twin groups and {G}rothendieck's
  cartographical groups.
\newblock {\em J. Algebra}, 530:215--234, 2019.

\bibitem{FT79}
R.~Fenn and P.~Taylor.
\newblock Introducing doodles.
\newblock In {\em Topology of low-dimensional manifolds ({P}roc. {S}econd
  {S}ussex {C}onf., {C}helwood {G}ate, 1977)}, volume 722 of {\em Lecture Notes
  in Math.}, pages 37--43. Springer, Berlin, 1979.

\bibitem{GGS92}
M.~Gerstenhaber, A.~Giaquinto, and S.~D. Schack.
\newblock Quantum symmetry.
\newblock In {\em Quantum groups ({L}eningrad, 1990)}, volume 1510 of {\em
  Lecture Notes in Math.}, pages 9--46. Springer, Berlin, 1992.

\bibitem{GoMeRo19}
J.~{Gonz{\'a}lez}, J.~L. {Le{\'o}n-Medina}, and C.~{Roque}.
\newblock {Linear motion planning with controlled collisions and pure planar
  braids}.
\newblock {\em arXiv e-prints}, page arXiv:1902.06190, Feb. 2019.

\bibitem{Got18}
K.~{Gotin}.
\newblock {Markov theorem for doodles on two-sphere}.
\newblock {\em arXiv e-prints}, page arXiv:1807.05337, Jul 2018.

\bibitem{HK20}
N.~L. Harshman and A.~C. Knapp.
\newblock Anyons from three-body hard-core interactions in one dimension.
\newblock {\em Ann. Physics}, 412:168003, 18, 2020.

\bibitem{KNS19}
T.~{Kanta Naik}, N.~{Nanda}, and M.~{Singh}.
\newblock {Conjugacy classes and automorphisms of twin groups}.
\newblock {\em Forum Math.}, 2020.
\newblock Accepted.

\bibitem{KNS19-2}
T.~{Kanta Naik}, N.~{Nanda}, and M.~{Singh}.
\newblock {Some remarks on twin groups}.
\newblock {\em J. Knot Theory Ramifications}, 2020.
\newblock Accepted.

\bibitem{KasTur08}
C.~Kassel and V.~Turaev.
\newblock {\em Braid groups}, volume 247 of {\em Graduate Texts in
  Mathematics}.
\newblock Springer, New York, 2008.
\newblock With the graphical assistance of Olivier Dodane.

\bibitem{KauLam04}
L.~H. Kauffman and S.~Lambropoulou.
\newblock Virtual braids.
\newblock {\em Fund. Math.}, 184:159--186, 2004.

\bibitem{Kh96}
M.~Khovanov.
\newblock Real {$K(\pi,1)$} arrangements from finite root systems.
\newblock {\em Math. Res. Lett.}, 3(2):261--274, 1996.

\bibitem{Kh97}
M.~Khovanov.
\newblock Doodle groups.
\newblock {\em Trans. Amer. Math. Soc.}, 349(6):2297--2315, 1997.

\bibitem{MasHan03}
J.~C. Mason and D.~C. Handscomb.
\newblock {\em Chebyshev polynomials}.
\newblock Chapman \& Hall/CRC, Boca Raton, FL, 2003.

\bibitem{Mer99}
A.~B. Merkov.
\newblock Vassiliev invariants classify flat braids.
\newblock In {\em Differential and symplectic topology of knots and curves},
  volume 190 of {\em Amer. Math. Soc. Transl. Ser. 2}, pages 83--102. Amer.
  Math. Soc., Providence, RI, 1999.

\bibitem{Mer99-2}
A.~B. Merkov.
\newblock Vassiliev invariants of doodles, ornaments, etc.
\newblock {\em Publ. Inst. Math. (Beograd) (N.S.)}, 66(80):101--126, 1999.
\newblock Geometric combinatorics (Kotor, 1998).

\bibitem{Mer03}
A.~B. Merkov.
\newblock Vassiliev invariants classify plane curves and sets of curves without
  triple intersections.
\newblock {\em Mat. Sb.}, 194(9):31--62, 2003.

\bibitem{MosRo19}
J.~Mostovoy and C.~Roque-M\'{a}rquez.
\newblock Planar pure braids on six strands.
\newblock {\em J. Knot Theory Ramifications}, 29(1):1950097, 11, 2020.

\bibitem{V99}
V.~A. Vassiliev.
\newblock On finite order invariants of triple point free plane curves.
\newblock In {\em Differential topology, infinite-dimensional {L}ie algebras,
  and applications}, volume 194 of {\em Amer. Math. Soc. Transl. Ser. 2}, pages
  275--300. Amer. Math. Soc., Providence, RI, 1999.

\bibitem{Voe90}
V.~A. Voevodsky.
\newblock Flags and grothendieck cartographical group in higher dimensions.
\newblock {\em CSTARCI Math. Preprints}, 1990.

\bibitem{Wang10}
Z.~Wang.
\newblock {\em Topological quantum computation}, volume 112 of {\em CBMS
  Regional Conference Series in Mathematics}.
\newblock Published for the Conference Board of the Mathematical Sciences,
  Washington, DC; by the American Mathematical Society, Providence, RI, 2010.

\end{thebibliography}

\medskip
{\small \sc Instituto de Matemáticas Universidad de Valparaíso,

Gran Breta\~{n}a 1111, Valpara\'{i}so 2340000, Chile.

{\tt marcelo.flores@uv.cl}

{\tt juyumaya@uvach.cl}
}

\bigskip\smallskip
{\small \sc Instituto de Matem\'aticas de la 
Universidad Nacional Aut\'onoma de M\'exico

Le\'on No.2, Altos, Oaxaca de Ju\'arez 68000, M\'exico

{\tt bruno@im.unam.mx}

{\tt croque@im.unam.mx}}

\end{document}